\documentclass[11pt]{article}

\usepackage{latexsym,amsfonts,amssymb,amsmath,amsthm}

\usepackage{color,comment}

\begin{document}

\newtheorem{theorem}{Theorem}[section]
\newtheorem{proposition}{Proposition}[section]
\newtheorem{definition}{Definition}[section] 
\newtheorem{corollary}{Corollary}[section]
\newtheorem{lemma}{Lemma}[section]
\newtheorem{assumption}{Assumption}[section]
\newtheorem{remark}{Remark}[section]
\newtheorem{nota}{Notation}[section]
\numberwithin{equation}{section}

\newcommand{\stk}[2]{\stackrel{#1}{#2}}
\newcommand{\dwn}[1]{{\scriptstyle #1}\downarrow}
\newcommand{\upa}[1]{{\scriptstyle #1}\uparrow}
\newcommand{\nea}[1]{{\scriptstyle #1}\nearrow}
\newcommand{\sea}[1]{\searrow {\scriptstyle #1}}
\newcommand{\csti}[3]{(#1+1) (#2)^{1/ (#1+1)} (#1)^{- #1
 / (#1+1)} (#3)^{ #1 / (#1 +1)}}
\newcommand{\RR}[1]{\mathbb{#1}}

\newcommand{\rd}{{\mathbb R^d}}
\newcommand{\ep}{\varepsilon}
\newcommand{\rr}{{\mathbb R}}
\newcommand{\alert}[1]{\fbox{#1}}
\newcommand{\eqd}{\sim}
\def\p{\partial}
\def\R{{\mathbb R}}
\def\N{{\mathbb N}}
\def\Q{{\mathbb Q}}
\def\C{{\mathbb C}}
\def\l{{\langle}}
\def\r{\rangle}
\def\t{\tau}
\def\k{\kappa}
\def\a{\alpha}
\def\la{\lambda}
\def\De{\Delta}
\def\de{\delta}
\def\ga{\gamma}
\def\Ga{\Gamma}
\def\ep{\varepsilon}
\def\eps{\varepsilon}
\def\si{\sigma}
\def\Re {{\rm Re}\,}
\def\Im {{\rm Im}\,}
\def\E{{\mathbb E}}
\def\P{{\mathbb P}}
\def\Z{{\mathbb Z}}
\def\D{{\mathbb D}}
\newcommand{\ceil}[1]{\lceil{#1}\rceil}

\title{Persistence and time periodic positive solutions of doubly nonlocal Fisher-KPP equations in time periodic and space heterogeneous media}
\author{Jianping Gao\\
College of Mathematics and Econometrics\\
 Hunan University\\
  Changsha, Hunan 410082\\
People's Republic of China\\
 and\\
 Department of Mathematics and Statistics\\
 Auburn University,  AL 36849\\
 \\
Shangjiang Guo\\
College of Mathematics and Econometrics\\
 Hunan University\\
  Changsha, Hunan 410082\\
People's Republic of China\\
\\
and\\
\\
Wenxian Shen\\
Department of Mathematics and Statistics\\
Auburn University, AL 36849}
\date{}

\maketitle
\noindent{\bf Abstract.}
In this paper, we investigate the asymptotic dynamics of  Fisher-KPP equations with nonlocal dispersal operator and nonlocal reaction term in time periodic and space heterogeneous media. We first show the global existence and  boundedness of nonnegative solutions, and then obtain some sufficient conditions ensuring the uniform persistence. In particular, we study the existence, uniqueness and global stability of positive time periodic solutions under several different conditions.
\medskip

\noindent{\bf Key words.}
 Fisher-KPP equation, nonlocal diffusion, nonlocal reaction, uniform persistence, positive time periodic solution, time-space dependence.

\medskip
\noindent {\bf 2010 Mathematics Subject Classification.} { 35K57, 47G20, 47J35, 58D25, 92D25.}

\section{Introduction}

In this paper, we consider the following three nonlocal dispersal equations with nonlocal intra-specific competition in  time periodic and space heterogeneous media,
    \begin{equation*}\label{0}
  u_t=\int_{\mathbb{R}^N}J(y-x)u(y, t)dy-u(x, t)+u(x, t)f_1(x, t, u, G_1\ast u), \qquad  x\in \mathbb{R}^N,\eqno({\bf P_1})
    \end{equation*}
 \begin{equation*}\label{87}
     u_t=\int_{\Omega}J(y-x)u(y, t)dy-u(x, t)+u(x, t)f_2(x, t, u, G_2\ast u), \qquad  x\in \bar{\Omega}, \eqno(\bf P_2)
    \end{equation*}
     \begin{equation*}\label{88}
  \quad  u_t=\int_{\Omega}J(y-x)(u(y, t)-u(x, t))dy+u(x, t)f_3(x, t, u, G_3\ast u), \qquad  x\in \bar{\Omega},\eqno(\bf P_3)
    \end{equation*}
    where $N\geq 1$, $\Omega\subset \mathbb{R}^N$ is a smooth bounded domain,
    \begin{equation*}
       \begin{aligned}
& G_1*u=\int_{\mathbb{R}^N} G(y-x)u(y,t)dy,\\
& G_2*u=G_3*u=\int_{\Omega} G(y-x)u(y,t)dy,
 \end{aligned}
\end{equation*}
    and
        \begin{equation*}
       \begin{aligned}
& f_1(x, t, u, G\ast u)=a_1(x,t)-b_1(x,t)u(x, t)-c_1(x,t)G_1\ast u (x, t),\\
& f_2(x, t, u, G\ast u)=a_2(x,t)-b_2(x,t)u(x, t)-c_2(x,t)G_2\ast u (x, t),\\
& f_3(x, t, u, G\ast u)=a_3(x,t)-b_3(x,t)u(x, t)-c_3(x,t)G_3\ast u (x, t).
 \end{aligned}
\end{equation*}

Throughout this paper, we always assume

\begin{itemize}
   \item [($\textbf{A}_0$)] {\it $J(\cdot)$ and $G(\cdot)$ are $C^1$ nonnegative convolution kernels supported on the balls $B_{r_0}$ and $B_{r_1}$, respectively, where $0<r_1, r_0<\infty$ and $B_r$ is a ball centered at $0$ with radius $r$, and
         $$\int_{\mathbb{R}^N}J(z)dz=\int_{\mathbb{R}^N}G(z)dz=1.$$
The functions $a_i(x,t)$, $b_i(x,t)$ and $c_i(x,t)$ $(i=1,2,3)$ are nonnegative, continuous in $(x,t)$,   periodic in $t$ with period $T$, and $a_1(x+p_j\textbf{e}_j, t)=a_1(x, t)$,  $b_1(x+p_j\textbf{e}_j, t)=b_1(x, t)$ and  $c_1(x+p_j\textbf{e}_j, t)=c_1(x, t)$  for $p_j>0$, $\textbf{e}_j=(e_{1j}, e_{2j},\cdots, e_{Nj})$ $(j=1,2,\cdots N)$ and for $k, j=1,2,\cdots N$, $e_{kj} = 1$ if $k=j$ and $0$ if $k\neq j$.}
\end{itemize}

 For convenience, let
    \begin{equation*}
       \begin{aligned}
&\mathcal{L}_1[u]=\int_{\mathbb{R}^N}J(y-x)u(y, t)dy-u(x, t),\\
& \mathcal{L}_2[u]=\int_{\Omega}J(y-x)u(y, t)dy-u(x, t),\\
& \mathcal{L}_3[u]=\int_{\Omega}J(y-x)(u(y, t)-u(x, t))dy.
 \end{aligned}
\end{equation*}

Systems $(\textbf{P}_1)$,  $(\textbf{P}_2)$ and $(\textbf{P}_3)$ can be used to model the evolution of the population of a species with nonlocal internal dispersal and  nonlocal intra-specific competition. In such case,
 $u(x,t)$ represents the population density of the species at space location $x$ and time $t$;
 $\mathcal{L}_i[u]$  $(i=1,2,3)$ characterizes the long range interaction or movement of the species and dispersal kernel $J$ describes the probability that the species jump from one position to another;  $a_i(x, t)$ $(i=1,2,3)$ denotes the intrinsic growth rate function; the term  $-b_i(x, t)u$ $(i=1,2,3)$ describes the competition of the individuals of the species for the resources in the environment that only depends on the value of $u(x, t)$ at point $x$; the term $-c_i(x,t)G_i\ast u$ $(i=1,2,3)$  accounts for competition of the individuals
of the species for the resources in the environment that not only depends on the value of $u(x, t)$ at point $x$, but also on the value of $u$ in a neighborhood of $x$. For any fixed $t\geq 0$, if we confine the solution of $(\mathbf{P}_1)$ to the space $X_1$ of space periodic functions  (see Section \ref{102}), then we call $(\mathbf{P}_1)$ is of periodic boundary condition $u(x+p_j\textbf{e}_j, t)=u(x, t)$ for $x\in \mathbb{R}^N$. $(\mathbf{P}_2)$ is of Dirichlet type boundary condition $\int_{\mathbb{R}^N\backslash \Omega}J(y-x)u(y)dy=0$ for $x \in \bar{\Omega}$, which describes a hostile environment outside $\Omega$ and any individual that jumps outside dies instantaneously. $(\mathbf{P}_3)$ is of Neumann type boundary condition $\int_{\mathbb{R}^N\backslash \Omega} J(y-x)(u(y)-u(x))dy=0$ for $x\in \bar{\Omega}$, which means that the individuals may not enter or leave the domain $\Omega$.

 Observe that $(\textbf{P}_i)$ can be put in the following general form,
  \begin{equation}\label{84}
    u_t=\textit{A}_i(u)+uf_i(x, t, u, K_i(u)),\quad x\in \Omega_i,
    \end{equation}
  complemented with proper boundary conditions,
where
 \begin{equation}
 \label{omega-i-eq}
 { \Omega_i}=\begin{cases}
   \mathbb{R}^N, \ \text{for}\ i=1,\\
   \bar{\Omega}\quad \text{for}\ i=2,3,
 \end{cases}
 \end{equation}
and $\textit{A}_i$, $\textit{K}_i$ are linear operators with respect to $u$.
 In fact, for $i=1,2,3$,  if $\textit{A}_i(u)=\mathcal{L}_i[u]$  and $\textit{K}_i(u)=G_i\ast u$, then  (\ref{84}) becomes $(\mathbf{P}_i)$.

Equation \eqref{84} with various special $\textit{A}_i$ and $\textit{K}_i$ induces many important equations in literature. For example,
 for $i=1,2,3$, if $\textit{A}_i(u)=\Delta u$ and $\textit{K}_i(u)=u$, then (\ref{84}) gives rise to  the following reaction-diffusion equations,
\begin{equation}\label{95}
    \begin{aligned}
   & u_t=\Delta u+u(x, t)f_1(x, t, u), \qquad  x\in \mathbb{R}^N,
    \end{aligned}
    \end{equation}
 \begin{equation}\label{96}
   \left\{
      \begin{aligned}
   & u_t=\Delta u+u(x, t)f_2(x, t, u), \qquad  x\in \Omega,\\
   & u(x, t)=0, \qquad x \in \partial \Omega,
       \end{aligned}
        \right.
    \end{equation}
    and
     \begin{equation}\label{97}
    \left\{
      \begin{aligned}
   & u_t=\Delta u+u(x, t)f_3(x, t, u), \qquad  x\in \Omega,\\
   & \frac{\partial u}{\partial n} (x, t)=0, \qquad x \in \partial \Omega;
       \end{aligned}
        \right.
    \end{equation}
 if $\textit{A}_i(u)=\Delta u$ and $\textit{K}_i(u)=G_i\ast u$, then (\ref{84}) induces the following nonlocal reaction-diffusion equations,
 \begin{equation}\label{89}
    \begin{aligned}
   & u_t=\Delta u+u(x, t)f_1(x, t, u, G_1\ast u), \qquad  x\in \mathbb{R}^N,
    \end{aligned}
    \end{equation}
 \begin{equation}\label{90}
   \left\{
      \begin{aligned}
   & u_t=\Delta u+u(x, t)f_2(x, t, u, G_2\ast u), \qquad  x\in \Omega,\\
   & u(x, t)=0, \qquad x \in \partial \Omega,
       \end{aligned}
        \right.
    \end{equation}
    and
     \begin{equation}\label{91}
    \left\{
      \begin{aligned}
   & u_t=\Delta u+u(x, t)f_3(x, t, u, G_3\ast u), \qquad  x\in \Omega,\\
   & \frac{\partial u}{\partial n} (x, t)=0, \qquad x \in \partial \Omega;
       \end{aligned}
        \right.
    \end{equation}
 if $\textit{A}_i(u)=\mathcal{L}_i[u]$ and $\textit{K}_i(u)=u$, then  (\ref{84}) becomes the following nonlocal dispersal equations,
  \begin{equation}\label{92}
    \begin{aligned}
   & u_t=\int_{\mathbb{R}^N}J(y-x)u(y, t)dy-u(x, t)+u(x, t)f_1(x, t, u), \qquad  x\in \mathbb{R}^N,
    \end{aligned}
    \end{equation}
 \begin{equation}\label{93}
    u_t=\int_{\Omega}J(y-x)u(y, t)dy-u(x, t)+u(x, t)f_2(x, t, u), \qquad  x\in \bar{\Omega},
    \end{equation}
    and
     \begin{equation}\label{94}
    u_t=\int_{\Omega}J(y-x)(u(y, t)-u(x, t))dy+u(x, t)f_3(x, t, u), \qquad  x\in \bar{\Omega}.
    \end{equation}

Equations (\ref{95})-(\ref{97}) are widely used to study the population dynamics of single species with local diffusion and local intra-specific competition.  Equations (\ref{89})-(\ref{91}) are called nonlocal reaction-diffusion equations which are used to model the dynamics of single species with local diffusion and nonlocal intra-specific competition. The reader { is referred} to \cite{furter1989local} for the introduction of nonlocal reaction term in population biology. Equations (\ref{92})-(\ref{94}) can be used to investigate the population dynamics of single species with only nonlocal dispersal operator. For the background of nonlocal dispersal operator, {the reader is referred} to \cite{andreu2010nonlocal,hutson2006non} and the references therein.

Among the central dynamical issues for $(\textbf{P}_1)$-$(\textbf{P}_3)$ (resp. (\ref{95})-(\ref{97}), (\ref{89})-(\ref{91}), (\ref{92})-(\ref{94})) are global existence of solutions with given nonnegative initial functions; uniform persistence of the population;
 existence, uniqueness and stability of positive time periodic solutions; spreading speeds and traveling wave solutions of  $(\textbf{P}_1)$
 (resp. (\ref{95}), (\ref{89}), (\ref{92})).
 There exist many works on these central dynamical issues. In the following, we recall some relevant works.

First,  models (\ref{95})-(\ref{97}) have been { extensively} investigated.  Model (\ref{95}) in time homogeneous case has been studied in the two pioneering works of Fisher \cite{fisher1937wave}  and Kolmogorov et al. \cite{KPP1937}, in which a well-known propagation phenomenon{ was obtained}. The reader is referred to  \cite{berestycki2005analysis,berestycki2005analysisb} for the study of persistence and pulsating
traveling fronts for (\ref{95}) with general local dispersal operator and general local reaction term in space periodic media.  Nadin in \cite{nadin2010existence,nadin2007reaction,nadin2009traveling} investigated (\ref{95}) with general local dispersal operator and general local reaction term in space-time periodic media, including the  existence, uniqueness and stability of periodic solutions 
and properties of traveling wave solutions.
   The reader is referred to the book \cite{hess1991periodic} and references therein for the study of the existence and stability properties of time periodic solution and asymptotic behavior of the solution of initial value problem associated with (\ref{96})-(\ref{97})  in more general case of local dispersal operator and local reaction term. From the papers mentioned  in the above, it is known that the properties of positive periodic solutions of (\ref{95})-(\ref{97}) in the general case are determined by the sign of the principal eigenvalue of the corresponding linearized equation at the null state.

 Second, models (\ref{89})-(\ref{91}) have been studied mainly  in the  time independent case. Most existing works are on (\ref{89}) with $b(x, t)\equiv 0$.  Hamel and Ryzhik \cite{hamel2014nonlocal} studied some properties of solutions of (\ref{89}) with $b(x, t)\equiv 0$ and $N=1$, including existence of non-constant  space periodic steady state, spreading speed and global boundedness  for the solution of the Cauchy problem. The study of (\ref{89}) with $b(x, t)\equiv 0$ and $N>1$ has been given in \cite{berestycki2009non}, in which the steady states and  traveling wave solutions connecting these states have been investigated. For more works on (\ref{89}) with $b(x, t)\equiv 0$, the reader{ is  referred}  to
\cite{alfaro2012rapid,gourley2000travelling,genieys2006pattern},  and for more works on nonlocal reaction-diffusion model on the whole space, the reader is  referred to \cite{apreutesei2010spatial, apreutesei2009travelling,billingham2003dynamics,deng2008nonlocal,deng2015global}. For the study of (\ref{90})-(\ref{91}) in time independence case, we refer the reader to \cite{bian2017global, correa2011some,delgado2016study,sun2013existence}, in which the authors applied the bifurcation theory and monotonicity methods to study the existence and stability of steady states.

Third, models (\ref{92})-(\ref{94}) have been  recently studied  with quite general reaction term. In the case { that}  $f_1(x, t, u)=f_1(x, u)$ is spatially periodic, the existence, uniqueness and stability of positive stationary solutions of  (\ref{92}) have been studied in \cite{coville2008existence,shen2012stationary,shen2010spreading}. Berestycki, Coville, and Vo \cite{berestycki2016persistence} employed the principal eigenvalue theory and obtained a persistence criteria for (\ref{92})  with the reaction term  that is time independent, but  depends on  $x$.  For the study of traveling waves and spreading properties of (\ref{92}) in time independent case, the reader is referred to \cite{coville2005propagation,shen2010spreading,shen2012traveling}.  For the time periodic case, Rawal and Shen \cite{Rawal2012} studied the principal eigenvalue theory of operators $\mathcal{L}_i[u]$ ($i=1,2,3$) and furthermore, applied the principal eigenvalue theory and the tool of part metric to obtain the existence, uniqueness and global stability of positive time periodic solutions for  (\ref{92})-(\ref{94}). In particular, Rawal and Shen \cite{Rawal2012} showed that the properties of positive time periodic solutions are determined by the sign of the principal spectrum point of the corresponding linearized equation of  (\ref{92})-(\ref{94})  at the null state. We  refer the reader to \cite{NWA2015} for the study of spreading properties and traveling waves of (\ref{92}) in time and space periodic case and to \cite{Shen2017} for the study of properties of transition waves and positive entire solutions of (\ref{92}) in general time and space dependence. For the study of other aspects of nonlocal dispersal models, the reader is referred  to \cite{shen2016spectral,sun2017periodic} and the references therein.

There exist some recent works on $(\textbf{P}_1)$-$(\textbf{P}_2)$ when $a_i(x,t)$, $b_i(x,t)$, and $c_i(x,t)$ are constant functions, but there is no work on $(\textbf{P}_3)$ even in the homogenous case. For $(\textbf{P}_1)$ with both local and nonlocal interaction in time and space independent case, we  refer the reader to \cite{finkelshtein2018doublyb} for the study of the existence and properties of traveling wave solutions, to \cite{finkelshtein2018doublya} for the study of the front propagation, to \cite{kuehn2018pattern} for the study of the pattern formation in one dimensional space. For the study of $(\textbf{P}_1)$ with only nonlocal interaction, the reader is referred  to \cite{finkelshtein2015traveling,finkelshtein2017global} and the references therein. The dynamics of $(\textbf{P}_2)$ recently has been investigated in \cite{ma2018dynamics}.


However, there is little study  on the central dynamical issues for $(\textbf{P}_1)$-$(\textbf{P}_3)$ with
 $a_i(x,t)$, $b_i(x,t)$ and $c_i(x,t)$ being non-constant functions. The difficulties for the study of $(\mathbf{P}_1)$-$(\mathbf{P}_3)$ lie in many aspects. One of them is the lack of compactness and regularities of the solutions of nonlocal dispersal evolution equations (which do not arise in the study of (\ref{95})-(\ref{97}) and (\ref{89})-(\ref{91})); another is the lack of standard comparison principle (which do not arise in the study of (\ref{95})-(\ref{97}) and (\ref{92})-\ref{94}). Moreover, the space and time dependence of the coefficients gives rise to some additional difficulties.

The objective of the current paper is to study the dynamical behaviors  of $(\mathbf{P}_1)$-$(\mathbf{P}_3)$, including persistence and the properties of positive time periodic solutions of $(\mathbf{P}_1)$-$(\mathbf{P}_3)$.
 Mainly, { under proper conditions on   $J$, $G$, and $a_i,b_i,c_i$}, we will prove
\begin{itemize}
  \item global existence and boundedness of  solutions to  $(\mathbf{P}_1)$-$(\mathbf{P}_3)$ with nonnegative initial functions (see Theorem \ref{98});

  \item uniform persistence of systems $(\mathbf{P}_1)$-$(\mathbf{P}_3)$ (see Theorem \ref{73});

      \item existence, uniqueness and global stability of positive time periodic solutions of $(\mathbf{P}_1)$-$(\mathbf{P}_3)$ (see Theorem \ref{104}).
\end{itemize}
We will study the spreading speeds  and traveling wave solutions of
$(\textbf{P}_1)$ in our future works.

The rest of the paper is organized as follows. In Section \ref{102}, we introduce some standing notations, assumptions, and definitions,
 and state our main results.  We devote Section \ref{52}-\ref{76} to the proofs of the main results.

\section{Notations, assumptions, definitions and main results}\label{102}

In this section, we introduce some standing notations, assumptions and definitions,
 and state the main results of the paper.

\subsection{Notations}

In this subsection, we introduce some standing notations.

 Let
 $$D_i=\begin{cases}
   [0, p_1]\times [0, p_2] \times \cdots \times [0, p_N], \quad \,\, \ \text{for}\ i=1,\\
   \bar{\Omega},\quad \quad \quad \quad \quad \quad \quad \quad \quad \quad \quad \quad \quad \,\,  \text{for}\ i=2,3.
 \end{cases}$$
Let
 \begin{equation*}
\begin{aligned}
     & a_{iM}=\max_{D_i\times [0, T]} a_i(x,t),\quad b_{iM}=\max_{D_i\times [0, T]} b_i(x, t), \quad c_{iM}=\max_{D_i\times [0, T]} c_i(x,t),\\
   & a_{iL}=\min_{D_i\times [0, T]} a_i(x,t),\quad b_{iL}=\min_{D_i\times [0, T]} b_i(x, t), \quad c_{iL}=\min_{D_i\times [0, T]} c_i(x,t),
\end{aligned}
 \end{equation*}
 and
\begin{equation}\label{111}
\begin{cases}
    g_{i,m}=\inf \limits_{x\in \Omega_i}\int_{\Omega_i} G(y-x) dy, \quad g_{i,M}=\sup \limits_{x\in \Omega_i}\int_{\Omega_i} G(y-x) dy,\cr
    j_{i,m}=\inf \limits_{x\in \Omega_i}\int_{\Omega_i} J(y-x) dy.
    \end{cases}
    \end{equation}

We  denote $|\cdot|$ the norm in $\mathbb{R}$ and $\|\cdot\|$ the norm in $\mathbb{R}^N$, and define the following spaces:
$$\hat{X}_1=\big{\{}u\in C(\mathbb{R}^N, \mathbb{R}): u\ \text{is uniformly continuous and bounded}\big{\}}$$
with norm $\|u\|_{\hat{X}_1}=\sup \limits_{x\in \mathbb{R}^N} |u(x)|$;
$$X_{1}=\big{\{}u\in \hat{X}_1: u(\cdot+p_i\textbf{e}_i)=u(\cdot)\big{\}}$$
with norm $\|u\|_{X_1}=\sup \limits_{x\in D_1} |u(x)|$;
$$\hat{X}_i=X_{i}=C(\bar{\Omega}, \mathbb{R}),\quad i=2,3
$$
 with norm $\|u\|_{\hat{X}_i}=\|u\|_{X_i}=\sup \limits_{x\in \bar{\Omega}} |u(x)|$;  and
$$X_i^+(\hat{X}_i^+)=\big{\{}u\in X_i(\hat{X}_i): u\geq 0 \big{\}}, \quad i=1,2,3,$$
$$X_i^{++}=\big{\{}u\in X_i^{+}: u(x)>0 \ \forall \ x\in \Omega_i \big{\}}, \quad i=1,2,3.$$
For given $i=1,2,3$, the solution $u_i(x,t)$ of $(\mathbf{P}_i)$ with initial value $u_i(\cdot,0)=u_0\in \hat{X}_i$, if it exists, is denoted by $u_i(x, t; u_0)$.

 \subsection{Definitions}

In this subsection, we introduce the definition of sup- and sub- solutions of $(\mathbf{P}_1)$-$(\mathbf{P}_3)$ and the definition of uniform
persistence of $(\mathbf{P}_1)$-$(\mathbf{P}_3)$.

\begin{definition}\label{23}
For given $i=1,2,3$, a pair of positive bounded continuous functions $\overline{U}_i(x, t)$ and $\underline{U}_i(x, t)$ on ${ \Omega_i}\times [0,  \infty)$ are called {\rm a pair of sup- and sub- solutions of $(\mathbf{P}_i)$} if $\frac{\partial \overline{U}_i}{\partial t}$ and $\frac{\partial \underline{U}_i}{\partial t}$ exist and are continuous on ${ \Omega_i}\times [0,  \infty)$,
and
\begin{equation*}
    \left\{
    \begin{aligned}
   & \overline{U}_{it}-\mathcal{L}_i[\overline{U}_i] \geq \overline{U}_i[a_i(x,t)-b_i(x,t)\overline{U}_i-c_i(x,t)G_i\ast \underline{U}_i]\\
   & \underline{U}_{it}-\mathcal{L}_i[\underline{U}_i]\leq \underline{U}_i[a_i(x,t)-b_i(x,t)\underline{U}_i-c_i(x,t)G_i\ast \overline{U}_i]
    \end{aligned}
    \right.
    \end{equation*}
    for $(x, t)\in { \Omega_i}\times [0, \infty)$.
\end{definition}

 \begin{definition}
For given $i=1,2,3$, we call  system $(\mathbf{P}_i)$ is {\rm uniformly persistent} if for any $u_0\in\hat X_i^+$, $u_i(x,t;u_0)$ exists for all
$t\ge 0$,
and if there exists ${\eta}_i>0$, such that for any $u_0$ satisfies
  \begin{equation} \label{new-added-eq1}
  \begin{cases}
   u_0\in \hat{X}_i^+\ \text{with}\ \inf \limits_{x\in \mathbb{R}^N} u_0(x)>0,    & \quad \text{for}\ i=1,\\
   u_0\in \hat{X}_i^+\setminus \{0\}, & \quad \text{for}\ i=2,3,
  \end{cases}
\end{equation}
 there exists ${T}_i(u_0)>0$ such that
	$${\eta}_i\leq u_i(x, t; u_0) \quad \text{for all}\ x\in { \Omega_i},\ t\geq {T}_i(u_0).$$
\end{definition}

 \subsection{Assumptions}

 In this subsection, we introduce some standing assumptions and make some remarks on the assumptions.

We first give the following standing assumption.

\begin{itemize}
   \item [($\textbf{A}_1$)] {\it  For $i=1,2,3$,
        \begin{equation*}
\begin{aligned}
& 0<a_{iL}\leq a_{iM}<\infty,\\
& 0<b_{iL}\leq b_{iM}<\infty,\\
& 0\le c_{iL}\leq c_{iM}<\infty.
 \end{aligned}
 \end{equation*}
 For $i=2$,
 $$
 j_{i,m}-1+a_{iL}>0.
 $$}
 \end{itemize}

 \begin{remark}
 \label{A1-rk}
 Assumption $(\textbf{A}_1)$ gives sufficient conditions for the instability of the trivial solution $u\equiv 0$ and for the existence, uniqueness, and stability of strictly positive time periodic solutions of the equation
\begin{equation}
\label{112}
u_t=\mathcal{L}_i[u]+u[a_i(x, t)-b_i(x, t)u], \quad x\in { \Omega_i} \quad i=1,2,3.
\end{equation}
 In fact, we have the following lemma.
 \end{remark}

\begin{lemma}
\label{basic-lm}
  Assume  $(\textbf{A}_0)$ and $(\textbf{A}_1)$ and let $i=1,2,3$ be given. Then \eqref{112}
has exactly two  time periodic solutions, $u=0$ and $u=u_i^*(x,t)$ with { $\inf_{x\in\Omega_i,t\in\R}u_i^*(x,t)>0$}. Moreover, $u=0$ is linearly unstable and the positive time periodic solution $u_i^*(x,t)$ is globally asymptotically stable in the sense that
$$\|\hat{u}_i(\cdot, t; u_0)-u_i^*(\cdot,t)\|_{\hat{X}_i}\rightarrow 0, \quad  t\rightarrow \infty$$
for any initial value $u_0$ satisfying (\ref{new-added-eq1}),
where $\hat{u}_i(x, t; u_0)$ is a solution of (\ref{112}) with initial value $u_0$.
\end{lemma}
\begin{proof}
Using the standard comparison principle for (\ref{112}) and \cite[Proposition 3.3]{shen2016spectral}, the proof of the lemma   follows from the arguments in \cite[Theorem E]{Rawal2012}.
\end{proof}

Let $(\textbf{A}_2)$  be the following standing assumption.

\begin{itemize}
\item[$(\textbf{A}_2$)]  $a_i(t, x)-c_i(x, t)G_i*u_i^*(x, t)>0$ for $(x,  t)\in { \Omega_i}\times [0,\infty)$ for $i=1,3$,
and  $\int_{\Omega_i} J(y-x)dy-1+a_i(t, x)-c_i(x, t)G_i*u_i^*(x, t)>0$ for $(x,  t)\in { \Omega_i}\times [0,\infty)$ for $i=2$.
\end{itemize}

\begin{remark}
\label{A2-A3-rk}
\begin{itemize}
\item[(1)]
As it is pointed out in Remark \ref{A1-rk}, assumption $(\textbf{A}_1)$ gives sufficient conditions for the instability of the trivial solution $u\equiv 0$ of \eqref{112} and hence gives sufficient conditions for the instability of the trivial solution $u_i\equiv 0$ of $(\mathbf{P}_i)$.
In the case where the nonlocal reaction in  $(\mathbf{P}_i)$ is absent, that is, $c_i(x,t)\equiv 0$,  by Lemma \ref{basic-lm},
persistence occurs in $(\mathbf{P}_i)$ and $(\mathbf{P}_i)$ has a unique globally asymptotically stable strictly positive periodic solution.
In general, we will show that
Assumption $(\textbf{A}_1)$ together with $(\textbf{A}_2)$   implies the persistence in $(\mathbf{P}_i)$ (see Theorem \ref{73}).

\item[(2)]
 If $a_{iL}-\frac{c_{iM}a_{iM}}{b_{iL}}>0$ (resp. $j_{i,m}-1+a_{iL}-\frac{c_{iM}a_{iM}}{b_{iL}}g_{i,M}>0$, $a_{iL}-\frac{c_{iM}a_{iM}}{b_{iL}}g_{i,M}>0$), then $(\textbf{A}_2$) holds for $i=1$ (resp. $i=2$, $i=3$).

\item[(3)] For  given $i=1$, $2$, or $3$, assume $(\textbf{A}_0)$-$(\textbf{A}_2)$. It can be proved that
	 there exist two continuous positive periodic functions $\underline{U}^i(x, t)\le \overline{U}^i(x, t)$ such that
\begin{equation}
\label{new-two-species-eq1}
\begin{cases}
\overline{U}^i_t=\mathcal{L}_i[\overline{U}^i]+\overline{U}^i[a_i(x,t)-b_i(x,t)\overline{U}^i-c_i(x,t)G_i\ast \underline{U}^i],\quad { x\in\Omega_i,\,\, t\in\R}\cr
\underline {U}^i_t=\mathcal{L}_i[\underline {U}^i]+\underline {U}^i[a_i(x,t)-b_i(x,t)\underline {U}^i-c_i(x,t)G_i\ast \overline{U}^i], \quad { x\in\Omega_i,\,\, t\in\R}
\end{cases}
\end{equation}
(see Theorem \ref{73}).
If $c_i(t,x)\equiv 0$, it is clear that
$\overline{U}^i(x,t)=\underline{U}^i(x,t)=u_i^*(x,t)$ for all $x\in\Omega_i$, $t\in\R$, and $i=1,2,3$.

\item[(4)] Let \begin{equation}\label{99}
  \hat{u}_1^*=\frac{a_{iM}b_{1M}-a_{1L}c_{1L}}{b_{1L}b_{1M}-c_{1L}c_{1M}}, \quad \hat{u}_{*1}=\frac{a_{1L}b_{1L}-a_{1M}c_{1M}}{b_{1L}b_{1M}-c_{1L}c_{1M}},
  \end{equation}
and
\begin{equation}\label{103}
  \hat{u}_3^*=\frac{a_{3M}b_{3M}-a_{3L}c_{3L}g_{3,m}}{b_{3L}b_{3M}-c_{3L}c_{3M}g_{3,m}g_{3,M}}, \quad \hat{u}_{*3}=\frac{a_{3L}b_{3L}-a_{3M}c_{3M}g_{3,M}}{b_{3L}b_{3M}-c_{3L}c_{3M}g_{3,m}g_{3,M}}.
\end{equation}
If \begin{equation}\label{74}
\begin{cases}
   a_{iL}b_{iL}-a_{iM}c_{iM}>0 \ \text{for}\ i=1\\
    a_{iL}b_{iL}-a_{iM}c_{iM}g_{i,M}>0\quad \text{for}\ i=3,
 \end{cases}
	\end{equation}
then
$$
0<\hat u_{*i}\le \underline{U}^i(t,x)\le \overline{U}^i(t,x)\le \hat u_i^*\le \frac{a_{iM}}{b_{iL}},\quad \forall\,\, t\in\R,\,\, x\in { \Omega_i},\, \, { i=1,3}.
$$
\end{itemize}
\end{remark}

 Let  ($\textbf{A}_3$), ($\textbf{A}_4$)  and  ($\textbf{A}_5$)  be the following standing assumptions,

\begin{itemize}
	\item [($\textbf{A}_3$)] $r_0>r_1$ and $J_{m}> c_{iM}\frac{a_{iM}}{b_{iL}}G_{M}$, where $i=1,2,3$,
	\begin{equation}\label{HH}
	J_{m}=\inf \limits_{x\in B_{r_1}} J(x), \quad G_{M}=\sup \limits_{x\in B_{r_1}} G(x).
	\end{equation}
\end{itemize}

\begin{itemize}
	\item [($\textbf{A}_4$)] ($\textbf{A}_2$) holds  and for $i=1,2,3$,
\begin{equation}
\label{persistence-cond-eq}
h_1^i(t,x)+h_2^i(t,x)<0,
\end{equation}
 where
\begin{align*}
&h_1^i(x, t)=a_i(x, t)-2b_i(x, t)\underline{U}^i-c_i(x, t)G_i*\underline{U}^i,\\
&h_2^i(x, t)=c_i(x, t)\overline{U}^i,
\end{align*}
and $\overline{U}^i$ and $\underline{U}^i$ are as in Remark \ref{A2-A3-rk} (3).
\end{itemize}

\begin{itemize}
\item [($\textbf{A}_5$)] $a_1(x, t)\equiv a(t)$, $b_1(x, t)\equiv b(t)$, $c_1(x, t)\equiv c(t)$, and  $b_{1L}>c_{1M}$.
\end{itemize}

\begin{remark}
\label{A4-A5-rk}
\begin{itemize}
\item[(1)] In the case where the nonlocal reaction term is absent  in  $(\mathbf{P}_i)$, that is, $c_i(x,t)\equiv 0$, the occurrence of persistence in $(\mathbf{P}_i)$ implies the existence, uniqueness, and stability of strictly positive periodic solutions of  $(\mathbf{P}_i)$. In general, we will prove that ($\textbf{A}_1$) together with ($\textbf{A}_3$) or  ($\textbf{A}_1$) together with ($\textbf{A}_4$) implies the existence, uniqueness, and stability of strictly positive periodic solutions of $(\mathbf{P}_i)$ (see Theorem \ref{104} (1)(2)), and that ($\textbf{A}_1$) together with ($\textbf{A}_5$) implies the existence, uniqueness, and stability of strictly positive periodic solutions of $(\mathbf{P}_1)$ (see Theorem \ref{104} (3)).

\item[(2)] Let $\hat{u}_i^*$ and   $\hat{u}_{*i}$ be as in Remark \ref{A2-A3-rk}(4). If ($\textbf{A}_2$) and \eqref{74} hold,
and
\begin{equation}\label{108}
a_{iM}-2b_{iL}\hat{u}_{*i}-c_{iL}g_{m}\hat{u}_{*i}+c_{iM}\hat{u}_i^*<0,
\end{equation}
then ($\textbf{A}_4$) holds for $i=1,3$.
\end{itemize}
\end{remark}

\subsection{Main results}

In this subsection, we state the main results of this paper.

 First, the following theorem includes  our main results on the global existence and boundedness   of nonnegative solutions for $(\mathbf{P}_1)$-$(\mathbf{P}_3)$.

 \begin{theorem}[Global existence and boundedness]\label{98} { Assume $(\textbf{A}_0)$, and
 let $i=1,2,3$ be given}.
 \begin{itemize}
 \item[(1)]
 $(\mathbf{P}_i)$ with initial value $u_0\in \hat{X}_i^+$ has a global solution $u_i(x,t; u_0)\in \hat{X}_i^+$, and
  if $u_0\in \hat{X}^+_i$ and $u_0 \not \equiv 0$, then $u_i(x, t; u_0)>0$ for $x\in \Omega_i$ and $t>0$. In addition, if $(\textbf{A}_1)$ holds,  then for any $u_0\in \hat{X}_i^+$  and any $M_i>\max\{\|u_0\|_{\hat X_i}, { \frac{a_{iM}}{b_{iL}}}\}$,  $\|u_i(\cdot, t; u_0)\|_{\hat{X}_i}\leq M_i, \ t\in[0,\infty)$.

\item[(2)]
  Suppose that $\overline{U}_i$ and $\underline{U}_i$ are  a pair of sup- and sub- solutions of $(\mathbf{P}_i)$ and $\overline{U}_i(\cdot,t)$,
   $\underline{U}_i(\cdot,t)\in \hat{X}_i^+$ for any $t\geq 0$. Then for any $u_0\in \hat{X}_i^+$ satisfying $\underline{U}_i(x, 0)\le u_0(x)\le \overline{U}_i(x, 0)$ for $x\in\Omega_i$, $(\mathbf{P}_i)$ admits a solution $u_i(x, t; u_0)$ on ${ \Omega_i}\times [0,\infty)$ which satisfies
  $$\underline{U}_i(x, t)\leq u_i(x, t; u_0)\leq \overline{U}_i(x, t), \quad (x, t)\in { \Omega_i}\times [0,\infty).$$
  \end{itemize}
 \end{theorem}

Next, we state our main results on the uniform persistence of solutions of $(\mathbf{P}_i)$, $i=1,2,3$.

\begin{theorem}[Uniform Persistence]\label{73}
	For each fixed $i=1$, $2$ or $3$, assume that $(\textbf{A}_0)$-$(\textbf{A}_2)$ hold. Then persistence occurs in
$(\mathbf{P}_i)$. More precisely,
 there exist two continuous positive periodic functions $\underline{U}^i(x, t)\le \overline{U}^i(x, t)$ such that  \eqref{new-two-species-eq1}
holds,
and	  for any $\epsilon>0$ small enough
 and any initial value $u_0$ satisfying (\ref{new-added-eq1}),
 there exists $t^i_{\varepsilon, u_0}$ such that
	\begin{equation}\label{66}
	0<\underline{U}^i{ (x,t)}-\varepsilon\leq u_i(x, t; u_0)\leq \overline{U}^i{ (x,t)}+\varepsilon
	\end{equation}
	for all $x\in {\Omega_i}$ and $t>t^i_{\varepsilon, u_0}$. Moreover, $\underline{U}^i(x,0)\leq u_0\leq \overline{U}^i(x,0)$ for $x\in\Omega_i$ implies
	\begin{equation}\label{67}
	\underline{U}^i(x,t)\leq u_i(x, t; u_0)\leq \overline{U}^i(x,t),\quad \forall\,\, x\in\Omega_i,\,\, t>0.
	\end{equation}
\end{theorem}

In the following, we state our main results on the existence, uniqueness and stability of positive time periodic solution of $(\mathbf{P}_1)$-$(\mathbf{P}_3)$.

\begin{theorem}[Positive time periodic solution]\label{104}
Let $i=1,2,3$ be given.
	\begin{itemize}
		\item [$(1)$] Assume that ($\textbf{A}_0$)-($\textbf{A}_1$) and ($\textbf{A}_3$) hold. Then
		$(\mathbf{P}_i)$ has exactly one time periodic solution $u^i_P(\cdot, t)\in \hat{X}_i^{++}$. Moreover, $u^i_P(\cdot, t)$ is globally asymptotically stable in the sense that
$$\|u_i(\cdot, t; u_0)-u^i_P(\cdot, t)\|_{\hat{X}_i}\rightarrow 0, \quad  t\rightarrow \infty$$
 for any { $u_0\in\hat X_i^+$ satisfying  \eqref{new-added-eq1}}

	\item [$(2)$] In addition to conditions ($\textbf{A}_0$)-($\textbf{A}_1$) and ($\textbf{A}_4$),
assume that kernel functions $J(\cdot)$ and $G(\cdot)$ are symmetric with respect to $0$.
Then $(\mathbf{P}_i)$ has exactly one time periodic solution $U_i^*(\cdot, t)\in X_i^{++}$. Moreover, $U_i^*(\cdot, x)$ is globally asymptotically stable in the sense that
		$$\|u_i(\cdot, t; u_0)-U_i^*(\cdot, t)\|_{X_i}\rightarrow 0, \quad  t\rightarrow \infty$$
 for any  $u_0\in  X_i^+ \setminus \{0\}$ .

	\item [$(3)$] Assume that ($\textbf{A}_0$)-($\textbf{A}_1$) and   ($\textbf{A}_5$) hold.  Then $(\mathbf{P}_1)$ has exactly one spatially homogeneous positive time periodic solution $\phi_1^*(t)$. Moreover, $\phi_1^*(t)$ is globally asymptotically stable in the sense that for any $u_0\in \hat{X}_1$ with $\inf \limits_{x\in \mathbb{R}^N} u_0(x)>0$,
\begin{equation}\label{83}
\|u_1(\cdot, t; u_0)-\phi_1^*(t)\|_{\hat{X}_1}\rightarrow 0, \quad  t\rightarrow \infty.
\end{equation}
	\end{itemize}
\end{theorem}

\begin{remark}
\label{periodic-solution-rk}
\begin{itemize}
\item[(1)]  By Theorem  \ref{104}, assumptions ($\textbf{A}_0$) and ($\textbf{A}_1$) together with ($\textbf{A}_3$) or ($\textbf{A}_4$)
imply that the persistence occurs in ($\mathbf{P_i}$), and  assumptions ($\textbf{A}_0$) and ($\textbf{A}_1$) together with ($\textbf{A}_5$)
imply that the persistence occurs in ($\mathbf{P_1}$).

\item[(2)] If $c_i(t,x)\equiv 0$, then \eqref{persistence-cond-eq} becomes
$$
a_i(t,x)-2 b_i(t,x)u_i^*(t,x)<0,\quad \forall\,\, t\in\R,\,\, x\in  { \Omega_i}.
$$
In view of Theorem \ref{73} and Remark \ref{A2-A3-rk} (3), however, we see that  \eqref{persistence-cond-eq} is not needed for Theorem \ref{104} (2) to hold in the case that $c_i(t,x)\equiv 0$.
\end{itemize}
\end{remark}

Before ending this section, we present several properties about the solutions of (\ref{92})-(\ref{94})  and a lemma which will be used in the proof of the main results. Let $i=1,2,3$  be given, for $u$, $v\in X_i$, we define
$$u\leq v (u \geq v) \ \text{if} \ v-u\in  X_i^+ (u-v\in  X_i^+),$$
and for $u$, $v\in \hat{X}_i$, we define
$$u\leq v (u \geq v) \ \text{if} \ v-u\in  \hat{X}_i^+ (u-v\in  \hat{X}_i^+).$$

Recall that $\hat{u}_1(x,t;u_0)$(resp. $\hat{u}_2(x,t;u_0)$, $\hat{u}_3(x,t;u_0)$) is the solution of (\ref{92})(resp. (\ref{93}), (\ref{94})) with $\hat{u}_1(\cdot, 0; u_0)=u_0(\cdot)\in \hat{X}_1$ (resp. $\hat{u}_2(\cdot, 0; u_0)=u_0(\cdot)\in \hat{X}_2$, $\hat{u}_3(\cdot, 0; u_0)=u_0(\cdot)\in \hat{X}_3$).
\begin{definition}\label{117}
  A continuous function $u(x, t)$ on $\Omega_1\times [0,\tau)$ is called a super-solution (sub-solution)
of (\ref{92}) if for any $x\in \R^N$, $u(x, t)$ is differentiable on $[0,\tau)$ and satisfies that for each $x\in \R^N$ and $t\in [0,\tau)$,
$$u_t\geq(\leq)\int_{\mathbb{R}^N}J(y-x)u(y, t)dy-u(x, t)+u(x, t)f_1(x, t, u).$$
\end{definition}
Similarly, we can define the super-solution and sub-solution of (\ref{93})-(\ref{94}).

\begin{proposition}[Comparison principle~\cite{Rawal2012,shen2010spreading}]\label{100}
Let $i=1,2,3$  be given.
  \begin{itemize}
    \item [$(1)$] If $u^1(x, t)$ and $u^2(x, t)$ are bounded sub- and super- solution of (\ref{92}) (resp. (\ref{93}), (\ref{94})) on $[0, \tau)$, respectively, and satisfy that $u^1(\cdot, 0)\leq u^2(\cdot, 0)$, then $u^1(\cdot, t)\leq u^2(\cdot, t)$ for $t\in [0, \tau)$.
    \item [$(2)$] For every $u_0 \in \hat{X}^+_i$, $\hat{u}_i(x,t;u_0)$ exists for all $t\geq 0$.
    \item [$(3)$] If $u^1$, $u^2\in \hat{X}_i^+$, $u^1\leq u^2$ and $u^1\not \equiv u^2$, then $\hat{u}_i(x, t;u^1)<\hat{u}_i(x, t;u^2)$ for $x\in \Omega_i$ and  $t>0$.
  \end{itemize}
\end{proposition}

\begin{lemma}\label{116}
 For $i=1,2,3$ be given. Assume that a continuous function $\Phi$ with
 $$\inf \limits_{x\in \Omega_i, t\geq 0} \Phi(x, t)>-\infty \ \text{and} \ \Phi(x, 0)\geq 0 \ \text{for} \ x \in \Omega_i$$
  satisfies
 \begin{equation}\label{114}
\Phi_t\geq\int_{\Omega_i}J(y-x)\Phi(y, t)dy+h_1(x, t)\Phi(x, t)+h_2(x, t)G_i \ast \Phi (x, t),
\end{equation}
 where both $h_1$ and $h_2$ are continuous functions, $h_1(x, t)>0$ for $(x, t)\in \Omega_i\times [0, +\infty)$ and $h_{1M}:=\sup \limits_{x\in \Omega_i, t\geq 0} h_1(x, t) <+\infty$.
 \begin{itemize}
   \item[$(1)$] If $h_2(x, t)<0$ for $(x, t)\in \Omega_i\times [0, +\infty)$, $h_{2m}:=\inf \limits_{x\in \Omega_i, t\geq 0} h_2(x, t)>-\infty$,  and for $r_0>r_1$,
   \begin{equation}\label{aa}
  J_m+G_Mh_{2m}>0,
   \end{equation}
       where $r_0$, $r_1$, $J_m$ and $G_M$ are defined in (\ref{HH}),
       then $\Phi(x, t)\geq 0$ for $(x, t)\in \Omega_i\times [0, +\infty)$.  Moreover, if $\Phi$ satisfies
        \begin{equation}\label{FF}
\Phi_t=\int_{\Omega_i}J(y-x)\Phi(y, t)dy+h_1(x, t)\Phi(x, t)+h_2(x, t)G_i \ast \Phi (x, t),
\end{equation}
 and   $\Phi(x, 0)\not \equiv  0 \ \text{for} \ x \in \Omega_i$, then  $\Phi(x, t)> 0$ for $(x, t)\in \Omega_i\times [0, +\infty)$.

   \item[$(2)$] If $h_2(x, t)\geq 0$ for $(x, t)\in \Omega_i\times [0, +\infty)$ and $\sup \limits_{x\in \Omega_i, t\geq 0} h_2(x, t) <+\infty$, then $\Phi(x, t)\geq 0$ for $(x, t)\in \Omega_i\times [0, +\infty)$.
 \end{itemize}
\end{lemma}
\begin{proof}
  $(1)$ We only prove  the case $i=1$ because the other cases can be dealt with analogously. We first prove the first part of $(1)$. We claim that $\Phi(x, t)\geq 0$ holds on $(0,\ T_0)$, where  $T_0=\frac{1}{1+h_{1M}}$. If not, then there exist $x^1\in \mathbb{R}^N$ and $t^1\in (0, T_0)$ such that $\Phi(x^1, t^1)<0$, which implies that we can find $0<t^0<T_0$ such that
	$$\Phi_{inf}:=\inf \limits_{\mathbb{R}^N\times [0, t^0]} \Phi(x, t)<0.$$
	By the condition in this lemma, we have  $\Phi_{inf}>-\infty$. We can then extract two sequences $\{x_n\}\in \mathbb{R}^N$ and $\{t_n\}\in [0, t^0]$ such that
	$$\Phi(x_n, t_n)\rightarrow \Phi_{inf}, \quad n\rightarrow \infty.$$
	From (\ref{114}), we have
	\begin{align}\label{115}
	&\Phi(x_n, t_n)-\Phi(x_n, 0) \nonumber\\
&=\int_0^{t_n} \big{(}\int_{\mathbb{R}^N}J(y-x_n)\Phi(y, s)dy+h_1(x_n, s)\Phi(x_n, s)+h_2(x_n, s)G_1*\Phi(x_n, s)\big{)}ds \nonumber\\
	&\geq \int_0^{t_n} \big{[} \int_{\mathbb{R}^N}(J(y-x_n)+h_2(x_n, s)G(y-x_n)))\Phi(y, s)dy+h_{1M} \Phi_{inf}\big{]}ds \nonumber\\
	& \geq \int_0^{t_n} (1+h_{1M})\Phi_{inf}ds=t_n(1+h_{1M})\Phi_{inf} \nonumber\\
	& \geq t^0(1+h_{1M})\Phi_{inf}.
	\end{align}
	Since $\Phi(x, 0)\geq 0$, then letting $n\rightarrow \infty$ in (\ref{115}), we get
	$$\Phi_{inf}\geq t^0(1+h_{1M})\Phi_{inf}>\Phi_{inf},$$
	which leads to a contradiction. Thus, $\Phi(x, t)\geq 0$ for $t\in [0, T_0)$. By continuation, we obtain $\Phi(x, t)\geq 0$ for $(x, t)\in \mathbb{R}^N\times (0, \infty)$.

 Now we prove the second part. It follows from the above discussion that  $\Phi(x, t)\geq 0$ for $(x, t)\in \mathbb{R}^N\times (0, \infty)$.
   If there exist $\hat{t}>0$ and $\hat{x}_0\in \mathbb{R}^N$ such that $\Phi(\hat{x}_0, \hat{t})=0$, then using the fact that $\Phi(x, t)\geq 0$ for $(x, t)\in \mathbb{R}^N\times (0, \infty)$ and the continuously differential of $\Phi$ with respect to $t$, we obtain from (\ref{FF})
	\begin{equation}\label{113}
	\int_{\mathbb{R}^N}[J(y-\hat{x}_0)+h_2(\hat{x}_0,\hat{t})u_2(\hat{x}_0, \hat{t})G(y-\hat{x}_0)]{ \Phi(y, \hat{t})}dy=0.
	\end{equation}
	By assumptions $(A_0)$ and (\ref{aa}), it follows from (\ref{113}) that
	$$   \int_{B_{r_1}(\hat{x}_0)}[J(y-\hat{x}_0)+{ h_2(\hat{x}_0,\hat{t})}G(y-\hat{x}_0)]{ \Phi(y, \hat{t})}dy=0,$$
	and
	$$\int_{B_{r_0}(\hat{x}_0) \backslash B_{r_1}(\hat{x}_0)}J(y-\hat{x}_0) \Phi(y, \hat{t}) dy=0, $$
	which means that $\Phi(x, \hat{t})=0$ for $x\in B_{r_0}(\hat{x}_0)$.
	Similarly, we can choose any $\hat{x}_1\in \partial B_{\frac{r_1}{2}}(\hat{x}_0)$ such that $\Phi(x, \hat{t})=0$ for all $x\in B_{\frac{3}{2}r_1}(\hat{x}_1)$. Thus we have $\Phi(x, \hat{t})=0$ for all $x\in \mathbb{R}^N$. By general group theory,
		(\ref{FF}) has a unique solution $\tilde{u}(\cdot, t; u_0)\in \hat{X}_1$ with initial value $u_0\in \hat{X}_1$ for $t\in \mathbb{R}$. Therefore, $\Phi(x, t)=0$ for $(x, t)\in \mathbb{R}^N\times [0, \infty)$ which contradicts the condition $\Phi(x, 0)\not \equiv  0 \ \text{for} \ x \in \Omega_1$. Hence,  $\Phi(x, t)> 0$ for $(x, t)\in \mathbb{R}^N\times [0, +\infty)$.


  $(2)$ The proof of conclusion (1) is similar to that of $(1)$ and hence is omitted.
\end{proof}

\begin{remark}
\label{new-added-rk1}
{ By the arguments of Lemma  \ref{116}, if $\Phi(x,t)$ is not continuous in $x$, the conclusion
$\Phi(x,t)\ge 0$ in Lemma \ref{116}(1) (resp. in Lemma \ref{116}(2))  still holds.}
\end{remark}

 \section{Global existence and boundedness}\label{52}

 In this section, we study the global existence and boundedness of solutions of $(\mathbf{P}_i)$ with given initial functions
 in $\hat X_i$ and give a proof of Theorem \ref{98}.

 \begin{proof}[Proof of Theorem \ref{98} $(1)$]
 We only focus on the case $i=1$ because the other cases can be dealt with analogously.

   First, for given $t\in\R$ and $u\in\hat X_1$, define $Au$ and $F(t,u)$ as
 \begin{align*}
   &(Au)(x)=\int_{\mathbb{R}^N}J(y-x)u(y)dy-u(x),\\
   & F(t,u)(x)=u(x)[a_1(x,t)-b_1(x,t)u(x)-c_1(x,t)G_1*u(x)].
 \end{align*}
 It is clear that $-A$ is a linear bounded operator on $\hat X_1$. Hence $-A$ generates a uniformly continuous semigroup on $\hat X_1$.
 It is also clear that for given $t\in\R$ and $u\in \hat X_1$, $F(t,u)\in \hat X_1$. Moreover, it is not difficult to verify that
 the mapping $[\R\times \hat X_1\ni (t,u)\mapsto F(t,u)\in\hat X_1]$ is continuous in $t$ and locally Lipschitz continuous in $u\in\hat X_1$. Then by general semigroup theory (see \cite[Theorem 6.1.4, Corollary 4.2.6]{Pazy}),
for any $u_0\in \hat{X}_1$, there exists a $t_{\max} \leq \infty$ such that $(\mathbf{P}_1)$ has a unique  solution $u_1(x,t; u_0)\in \hat{X}_1$ on  $[0, t_{\max})$ satisfying $u_1(x,0;u_0)=u_0(x)$, and if $t_{\max}<\infty$, then
$$
\lim_{t\to t_{\max}-} \|u_1(\cdot,t;u_0)\|_{\hat{X}_1}=\infty.
$$

Next, we claim that, if initial value $u_0\in \hat{X}_1^+$, then  $u^*=u_1(\cdot, t; u_0)\in \hat{X}_1^+$ for all $t\in [0, t_{\max})$. In fact,
 set $u^*=u_1(x, t; u_0)$. Note that $u^*$ satisfies
    \begin{equation*}
    	\left\{
    	\begin{aligned}
    		& u^*_t=\int_{\mathbb{R}^N}J(y-x)u^*(y, t)dy-u^*(x, t)+u^*[a_1(x,t)-b_1(x,t)u^*-c_1(x,t)G_1\ast u^*], \\
    		&u^*(x,0)=u_0(x)\in \hat{X}^+_1.
    	\end{aligned}
    	\right.
    \end{equation*}
 Hence $u^*$ is a solution of
  \begin{equation*}
    \left\{
    \begin{aligned}
   &u_t=\int_{\mathbb{R}^N}J(y-x)u(y, t)dy-u(x, t)+u[a_1(x,t)-b_1(x,t)u^*-c_1(x,t)G_1\ast u^*], \\
   &u(x,0)=u_0(x)\in \hat{X}^+_1.
    \end{aligned}
    \right.
    \end{equation*}
Then it follows from Proposition \ref{100} that
 $u^*=u_1(\cdot, t;u_0)\ge 0$ for $t\in [0, t_{\max})$. In particular, if $u_0\in \hat{X}^+_1$ and $u_0 \not \equiv 0$, then $u_1(x, t; u_0)>0$ for $x\in \mathbb{R}^N$ and $t\in (0,t_{\max})$.

 { We now prove that, for any $u_0\in\hat X_1^+$,  $t_{\max}=\infty$.
Note that $u_1(\cdot, t;u_0)\ge 0$ for $t\ge 0$, then we have
\begin{align*}
  u^*_t &= \int_{\mathbb{R}^N}J(y-x)u^*(y, t)dy-u^*(x, t)+u^*[a_1(x,t)-b_1(x,t)u^*-c_1(x,t)G_1\ast u^*] \\
  &\leq \int_{\mathbb{R}^N}J(y-x)u^*(y, t)dy-u^*(x, t)+u^*[a_1(x,t)-b_1(x,t)u^*],
\end{align*}
which together with Definition \ref{117} implies that $u^*$ is a sub-solution of the following equation
\begin{equation}\label{1}
    \left\{
    \begin{aligned}
   & u_t=\int_{\mathbb{R}^N}J(y-x)u(y, t)dy-u(x, t)+u[a_1(x,t)-b_1(x,t)u], \\
   & u(x,0)=u_0(x)\in \hat{X}^+_1.
    \end{aligned}
    \right.
    \end{equation}
Obviously,  equation (\ref{1}) admits a solution which exists globally. This implies that $t_{\max}=\infty$.
 }

 {  Finally, if $(\mathbf{A}_1)$ holds,  as stated before, $u_1(t,\cdot;u_0)$ is a sub-solution of \eqref{1}.
It is then clear that for any $M_1>\max\{\|u_0\|_1, { \frac{a_{1M}}{b_{1L}}} \}$, we have
$u_1(x, t;u_0)\leq M_1$ for $x\in\R^N$ and $t\in\R^+$. The proof of Theorem \ref{98} $(1)$ is thus completed.}
\end{proof}

In order to prove Theorem \ref{98} $(2)$, we first prove the following lemma.

 \begin{lemma}\label{21}
For $i=1,2,3$, if $\overline{U}_i$ and $\underline{U}_i$ are a pair of  sup- and sub- solutions of $(\mathbf{P}_i)$ on ${\Omega_i}\times [0, \infty)$,  and $\overline{U}_i(x, 0)\geq \underline{U}_i(x, 0)$, then $\overline{U}_i(x,t)\geq \underline{U}_i(x,t)$ on ${ \Omega_i}\times (0,\infty)$.
\end{lemma}

 \begin{proof}
 We only consider the case $i=1$ because the other cases can be dealt with similarly.
Since $\overline{U}_1$ and $\underline{U}_1$ are  a pair of sup- and sub- solutions of $(\mathbf{P}_1)$ on $\mathbb{R}^N\times [0, \infty)$,  then there is $M_0>0$ such that
$$0\le \overline{U}_1(x, t)< { M_0}, \ 0\le \underline{U}_1(x, t)< { M_0} \ \text{for} \ (x, t)\in \mathbb{R}^N\times [0, \infty).$$
  Setting $\omega=e^{ht}(\overline{U}_1-\underline{U}_1)$, where $h$ is positive constant that will  be determined later, then $\omega(x, 0)\geq 0$ and
\begin{align}\label{18}
 \omega_t & =he^{ht}(\overline{U}_1-\underline{U}_1)+ e^{ht}(\overline{U}_{1t}-\underline{U}_{1t}) \nonumber \\
  &= \int_{\mathbb{R}^N}J(y-x)\omega(y, t)dy  +
 p(x,t)\omega{+ c_1(x, t)\overline{U}_1G_1*\omega},
\end{align}
where
$$p(x, t)=h+a_1(x, t)-1-b_1(x, t)(\bar U_1+\underline {U}_1)-c_1(x, t){G_1*\underline {U}_1}.$$
Choose $h$ large enough such that $p(x, t)\geq 0$ for $(x, t)\in \mathbb{R}^N\times [0,\infty)$,
then it follows from Lemma \ref{116} $(2)$ that $\omega(x, t)\geq 0$ for $(x, t)\in \mathbb{R}^N\times (0, \infty)$ which implies $\overline{U}_1(x,t)\geq \underline{U}_1(x,t)$ for $(x, t)\in \mathbb{R}^N\times [0, \infty)$.
 \end{proof}

\begin{remark}
\label{new-added-rk2}
{ By Remark \ref{new-added-rk1} and the arguments of Lemma  \ref{21}, if $\overline{U}_i(x,t)$ and $\underline{U}_i(x,t)$ are not continuous
in $x$, the conclusion $\overline{U}_i(x,t)\ge\underline{U}_i(x,t)$ in Lemma \ref{21} still holds.}
\end{remark}

 Now we give a proof of Theorem \ref{98} $(2)$.

 \begin{proof}[Proof of Theorem \ref{98} $(2)$]
 We only consider the case $i=1$ because the other cases can be dealt with similarly.

 First, let $0\leq \underline{U}_1\leq \overline{U}_1\leq M_0$ and choose $\bar{M}$ large enough such that
  $$\min \limits_{0\leq \eta\leq M_0} (a_{1L}-2b_{1M}\eta+\bar{M})\geq 0,$$
  then for any $\eta\in [0, M_0]$ and $(x, t)\in \mathbb{R}^N\times [0,\infty)$, we have
  $a_1(x, t)-2b_1(x, t)\eta+\bar{M}\geq 0$.

   Next, we construct two sequences $\{\overline{U}_1^k\}$ and $\{\underline{U}_1^k\}$ with $\overline{U}_1^0=\overline{U}_1$ and $\underline{U}_1^0=\underline{U}_1$ such that for $(x, t)\in \mathbb{R}^N\times [0,\infty)$,
   \begin{equation}\label{20}
    \left\{
    \begin{aligned}
   &\overline{U}_{1t}^k-\mathcal{L}_1[\overline{U}_{1}^k]=\overline{U}_{1}^{k-1}[a_1(x, t)-b_1(x, t)\overline{U}_{1}^{k-1}]-c_1(x, t)\overline{U}_{1}^kG_1\ast \underline{U}_{1}^{k-1}-\bar{M}(\overline{U}_{1}^k-\overline{U}_{1}^{k-1}),\\
   &\underline{U}_{1t}^k-\mathcal{L}_1[\underline{U}_{1}^k]=\underline{U}_{1}^{k-1}[a_1(x, t)-b_1(x, t)\underline{U}_{1}^{k-1}]-c_1(x, t)\underline{U}_{1}^kG_1\ast \overline{U}_{1}^{k-1}-\bar{M}(\underline{U}_{1}^k-\underline{U}_{1}^{k-1}),\\
   &\overline{U}_{1}^k(x, 0)=u_0(x),\\
   &\underline{U}_{1}^k(x, 0)=u_0(x).
    \end{aligned}
    \right.
    \end{equation}

   Using the general semigroup theory, we know that for any given $(\overline{U}_{1}^{k-1},\underline{U}_{1}^{k-1})$ with $\overline{U}_{1}^{k-1}(x, t)$, $\underline{U}_{1}^{k-1}(x, t)$ being continuous in $t$ and $\overline{U}_{1}^{k-1}(\cdot, t)$, $\underline{U}_{1}^{k-1}(\cdot, t)\in \hat{X}_1$ for $t\geq 0$, equation (\ref{20}) has exactly one solution $(\overline{U}_{1}^k,\underline{U}_{1}^k)$,  which implies for each $k\geq 1$, the sequences $\{\overline{U}_{1}^k\}$ and $\{\underline{U}_{1}^k\}$ are well-defined, and $\overline{U}_{1}^k(x, t),\ \underline{U}_{1}^k(x, t)$ are continuous in $t$ and $\overline{U}_{1}^k(\cdot, t),\ \underline{U}_{1}^k(\cdot, t)\in \hat{X}_1$ for $t\geq0$.

Second, we claim  that $\{\overline{U}_{1}^k\}$ and $\{\underline{U}_{1}^k\}$ have the following ordered relationship for $(x, t)\in \mathbb{R}^N\times [0,\infty)$,
    \begin{equation}\label{eq-order}
    \underline{U}_{1}(x, t)\leq \underline{U}_{1}^k(x, t)\leq \underline{U}_{1}^{k+1}(x, t)\leq \cdots \leq \overline{U}_{1}^{k+1}(x, t)\leq \overline{U}_{1}^k(x, t)\leq \overline{U}_{1}(x, t).
    \end{equation}

    In fact,
    let $\bar{\omega}_1=\underline{U}_{1}(x, t)-\underline{U}_{1}^1(x, t)$, then $\bar{\omega}_1$ satisfies
    \begin{equation*}
    \left\{
    \begin{aligned}
   & \bar{\omega}_{1t}\leq \mathcal{L}_1[\bar{\omega}_{1}]-[c_1(x, t)G_1\ast \overline{U}_{1}+\bar{M}]\bar{\omega}_{1},\\
   & \bar{\omega}_{1}(x, 0)=\underline{U}_{1}(x, 0)-\underline{U}_{1}^1(x, 0)\leq 0.
    \end{aligned}
    \right.
    \end{equation*}
    By Proposition \ref{100}, we have $\bar{\omega}_{1}\leq 0$, namely, $\underline{U}_{1}(x, t)\leq \underline{U}_{1}^1(x, t)$ for $(x, t)\in \mathbb{R}^N\times [0,\infty)$. Similarly, we also have $\overline{U}_{1}^1(x, t)\leq \overline{U}_{1}(x, t)$. Let $\hat{\omega}_1=\underline{U}_{1}^1(x, t)-\overline{U}_{1}^1(x, t)$, then we have 
    \begin{align*}
\hat{\omega}_{1t} & =\mathcal{L}_1[\hat{\omega}_1]+\underline{U}_{1}[a_1(x, t)-b_1(x, t)\underline{U}_{1}]-\overline{U}_{1}[a_1(x, t)-b_1(x, t)\overline{U}_{1}] \\
&\qquad +\bar{M}(\underline{U}_{1}-\overline{U}_{1})-c_1(x, t)\overline{U}_{1}^1G_1\ast (\overline{U}_{1}-\underline{U}_{1})-[c_1(x, t)G_1\ast \overline{U}_{1}+\bar{M}]\hat{\omega}_1\\
  &=\mathcal{L}_1[\hat{\omega}_1]+(\bar{M}+a_1(x, t)-b_1(x, t)\eta_1)(\underline{U}_{1}-\overline{U}_{1})\\
 &\qquad +\bar{M}(\underline{U}_{1}-\overline{U}_{1})-c_1(x, t)\overline{U}_{1}^1G_1\ast (\overline{U}_{1}-\underline{U}_{1})-[c_1(x, t)G_1\ast \overline{U}_{1}+\bar{M}]\hat{\omega}_1,
\end{align*}
where $\eta_1= \underline{U}_{1}+\overline{U}_{1}$. It follows from $\bar{M}+a_1(x, t)-b_1(x, t)\eta_1\geq 0$ and $\underline{U}_{1}\leq \overline{U}_{1}$ that
\begin{equation*}
    \left\{
    \begin{aligned}
   & \hat{\omega}_{1t}\leq \mathcal{L}_1[\hat{\omega}_1]-[c_1(x, t)G_1\ast \overline{U}_{1}+\bar{M}]\hat{\omega}_{1}\\
   & \hat{\omega}_{1}(x, 0)=\underline{U}_{1}^1(x, 0)-\overline{U}_{1}^1(x, 0)=0\leq 0,
    \end{aligned}
    \right.
    \end{equation*}
which together with the comparison principle (see Proposition \ref{100}) implies that $\underline{U}_{1}^1\leq \overline{U}_{1}^1$, Hence we get  that
$$\underline{U}_{1}(x, t)\leq \underline{U}_{1}^1(x, t)\leq \overline{U}_{1}^1(x, t)\leq \overline{U}_{1}(x, t).$$
Moreover, 
by the choice of $\bar{M}$ and using the facts that $\underline{U}_{1}(x, t)\leq \underline{U}_{1}^1(x, t)$ and $\overline{U}_{1}^1(x, t)\leq \overline{U}_{1}(x, t)$, we have
\begin{align*}
&\qquad \underline{U}_{1t}^1-\mathcal{L}_1[\underline{U}_{1}^1] \\
&=\underline{U}_{1}(a_1(x, t)-b_1(x, t)\underline{U}_{1})-c_1(x, t)\underline{U}_{1}^1G_1\ast \overline{U}_{1}-\bar{M}(\underline{U}_{1}^1-\underline{U}_{1})\\
&\leq \underline{U}_{1}^1(a_1(x, t)-b_1(x, t)\underline{U}_{1}^1-c_1(x, t)G_1\ast \overline{U}_{1}^1)+(\bar{M}+a_1(x, t)-2b_1(x, t)\eta_2)(\underline{U}_{1}-\underline{U}_{1}^1)\\
&\leq \underline{U}_{1}^1(a_1(x, t)-b_1(x, t)\underline{U}_{1}^1-c_1(x, t)G_1\ast \overline{U}_{1}^1),
\end{align*}
and
\begin{align*}
&\qquad \overline{U}_{1t}^1-\mathcal{L}_1[\overline{U}_{1}^1] \\
&=\overline{U}_{1}(a_1(x, t)-b_1(x, t)\overline{U}_{1})-c_1(x, t)\overline{U}_{1}^1G_1\ast \underline{U}_{1}-\bar{M}(\overline{U}_{1}^1-\overline{U}_{1})\\
&\geq \overline{U}_{1}^1(a_1(x, t)-b_1(x, t)\overline{U}_{1}^1-c_1(x, t)G_1\ast \underline{U}_{1}^1)+(\bar{M}+a_1(x, t)-2b_1(x, t)\eta_3)(\overline{U}_{1}-\overline{U}_{1}^1)\\
&\geq \overline{U}_{1}^1(a_1(x, t)-b_1(x, t)\overline{U}_{1}^1-c_1(x, t)G_1\ast \underline{U}_{1}^1),
\end{align*}
where $\underline{U}_{1}\leq \eta_2 \leq \underline{U}_{1}^1$ and $\overline{U}_{1}^1\leq  \eta_3 \leq \overline{U}_{1}$, which means that $\overline{U}_{1}^1$ and $\underline{U}_{1}^1$ are also a pair of sup- and sub- solutions of $(\mathbf{P}_1)$.
By the above procedure, for any $k=2,3,\cdots$, we have
$$\underline{U}_{1}^k\leq \underline{U}_{1}^{k+1}\leq \overline{U}_{1}^{k+1}\leq \overline{U}_{1}^{k}.$$
So by induction, we obtain the ordered relationship (\ref{eq-order}).  

\smallskip

Now, by the monotonicity of $\{\overline{U}_{1}^k\}$ and $\{\underline{U}_{1}^k\}$, we have two functions {$\overline{u}_1$} and $\underline{u}_1$ such that
{ $\overline{u}_1(x,t)$ is upper semi-continuous and $\underline{u}_1(x,t)$  is lower semi-continuous}, $\overline{U}_{1}^k\rightarrow {\overline{u}_1}$ and $\underline{U}_{1}^k\rightarrow \underline{u}_1$ pointswise in $\mathbb{R}^N\times [0,\infty)$ as $k\rightarrow\infty$, and {  $\overline{u}_1\geq \underline{u}_1$}. Observe that
\begin{align*}
\overline{U}_1^k(x,t)=&u_0(x)+\int_0 ^t \mathcal{L}_1(\overline{U}_1^k)(x,s)ds+\int_0^ t \overline{U}_1^k[a_1(x,s)-b_1(x,s)\overline{U}_1^k-c_1(x,s)G\ast \underline{U}_1^k]ds\\
& -\int_0^t M[\overline{U}_1^k(x,s)-\overline{U}_1^{k-1}(x,s)]ds,
\end{align*}
and
\begin{align*}
\underline {U}_1^k(x,t)=&u_0(x)+\int_0 ^t \mathcal{L}_1(\underline {U}_1^k)(x,s)ds+\int_0^ t \underline {U}_1^k[a_1(x,s)-b_1(x,s)\underline {U}_1^k-c_1(x,s)G\ast \bar {U}_1^k]ds\\
& -\int_0^t M[\underline {U}_1^k(x,s)-\underline {U}_1^{k-1}(x,s)]ds,
\end{align*}
then by the Dominated Convergence Theorem, we have
\begin{equation}\label{new-add-eq-1}
    \left\{
    \begin{aligned}
   & \overline{u}_1(x,t)= u_0(x)+\int_0 ^t \mathcal{L}_1{ (\overline{u}_1)(x,s)}ds+\int_0^ t { \overline{u}_1}[a_1(x,s)-b_1(x,s)\overline{u}_1-c_1(x,s)G\ast \underline{u}_1]ds,\\
   & \underline {u}_1(x,t)=u_0(x)+\int_0 ^t \mathcal{L}_1(\underline {u}_1)(x,s)ds+\int_0^ t \underline {u}_1[a_1(x,s)-b_1(x,s)\underline {u}_1-c_1(x,s)G\ast \overline{u}_1]ds.
    \end{aligned}
    \right.
    \end{equation}
which implies that { $\overline{u}_1(x,t)$ and $\underline{u}_1(x,t)$ are continuous in $t$. }

{ Note that there is $K>0$ such that
\begin{equation}
\label{new-add-eq0}
|\p_t \overline{U}_1^k(x,t)|\le K,\quad |\p_t\underline{U}_1^k(x,t)|\le K\quad \forall\,\, x\in\R^N,\,\, t\ge 0,\,\, k\geq 1.
\end{equation}
Hence $\overline{U}_1^k(x,t)$ and $\underline{U}_1^k(x,t)$ are continuous in $t\ge 0$ uniformly with respect to $k\ge 1$ and $x\in\R^N$.
This implies that $\int_{\R^N}J(y-x)\overline{U}_1^k(y,t)dy$, $\int_{\R^N} J(y-x)\underline{U}_1^k(y,t)dy$, $\int_{\R^N}G(y-x)\overline{U}_1^k(y,t)dy$, and $\int_{\R^N} G(y-x)\underline{U}_1^k(y,t)dy$ are continuous in $t\ge 0$ uniformly with respect to
$k\ge 0$ and $x\in\R^N$. It then follows that $\int_{\R^N}J(y-x)\overline{u}_1(y,t)dy$, $\int_{\R^N} J(y-x)\underline{u}_1(y,t)dy$, $\int_{\R^N}G(y-x)\overline{u}_1(y,t)dy$, and $\int_{\R^N} G(y-x)\underline{u}_1(y,t)dy$ are continuous in $t$.
This together with \eqref{new-add-eq-1} implies that $\overline{u}_1(x,t)$ and $\underline{u}_1(x,t)$ are differentiable
in $t$ for $t\ge 0$ and}
\begin{equation*}
    \left\{
    \begin{aligned}
   & \overline{u}_{1t}-\mathcal{L}_1[\overline{u}_1]= \overline{u}_1[a_1(x,t)-b_1(x,t)\overline{u}_1-c_1(x,t)G\ast \underline{u}_1],\quad x\in\R^N\\
   & \underline{u}_{1t}-\mathcal{L}_1[\underline{u}_1]= \underline{u}_1[a_1(x,t)-b_1(x,t)\underline{u}_1-c_1(x,t)G\ast \overline{u}_1],\quad x\in\R^N\\
   &\overline{u}_1(x, 0)=u_0(x),\quad x\in\R^N\\
   &\underline{u}_2(x, 0)=u_0(x),\quad x\in\R^N.
    \end{aligned}
    \right.
    \end{equation*}
By Remark \ref{new-added-rk2}, we have
$\overline{u}_1\leq \underline{u}_1$ and then $u(x,t)=\overline{u}_1(x,t)=\underline{u}_1(x,t)$ is both upper and lower semi-continuous and
hence is continuous. So $\overline{u}_1=\underline{u}_1=u_1(x, t; u_0)$ is a solution of $(\mathbf{P}_1)$. The proof is completed.
\end{proof}
\begin{remark}
\begin{itemize}

\item [$(1)$] For { $i=1$}, if we can find a small $\epsilon_1>0$ and a large $\bar{M}_1>0$ such that for $(x, t)\in {\mathbb{R}^N}\times [0,\infty)$,
\begin{equation*}
    \left\{
    \begin{aligned}
   & a_1(x,t)-b_1(x,t)\bar{M}_1-c_1(x,t)\epsilon_1\leq 0,\\
   & a_1(x,t)-b_1(x,t)\epsilon_1-c_1(x,t)\bar{M}_i\geq 0,
    \end{aligned}
    \right.
    \end{equation*}
    then  $\epsilon_1$ and $\bar{M}_1$ are a pair of sub- and sup- solutions of $(\mathbf{P}_1)$. Hence,
    $$Q_{\epsilon_1}^{\bar{M}_1}:=[\epsilon_1, \bar{M}_1]=\big{\{}u\in \hat X_1: \epsilon_1\leq u(x)\leq \bar{M}_1 \big{\}}$$
    is an invariant region of $(\mathbf{P}_1)$.

 \item [$(2)$]  For { $i=2$},  if we can find a small $\epsilon_2>0$ and a large $\bar{M}_2>0$ such that for $(x, t)\in \Omega\times [0,\infty)$,
\begin{equation*}
    \left\{
    \begin{aligned}
   & \int_\Omega J(y-x) dy-1 + a_2(x,t)-b_2(x,t)\bar{M}_2-c_2(x,t)\epsilon_2 \int_\Omega G(y-x) dy\leq 0,\\
   & \int_\Omega J(y-x)dy-1+a_2(x,t)-b_2(x,t)\epsilon_2-c_2(x,t)\bar{M}_2\int_\Omega G(y-x) dy\geq 0,
    \end{aligned}
    \right.
    \end{equation*}
then $\epsilon_2$ and $\bar{M}_2$ are a pair of sub- and sup- solutions of $(\mathbf{P}_2)$. Hence,
$$Q_{\epsilon_2}^{\bar{M}_2}:=[\epsilon_2, \bar{M}_2]=\big{\{}u\in \hat X_2: \epsilon_2\leq u(x)\leq \bar{M}_2 \big{\}}$$
is an invariant region of $(\mathbf{P}_2)$.
\item [$(3)$]  For $i=3$,  if we can find a small $\epsilon_3>0$ and a large $\bar{M}_3>0$ such that for $(x, t)\in \Omega\times [0,\infty)$,
\begin{equation*}
    \left\{
    \begin{aligned}
   &   a_3(x,t)-b_3(x,t)M_3-c_3(x,t)\epsilon_3 \int_\Omega G(y-x) dy\leq 0,\\
   & a_3(x,t)-b_3(x,t)\epsilon_3-c_3(x,t)M_3\int_\Omega G(y-x) dy\geq 0,
    \end{aligned}
    \right.
    \end{equation*}
then  $\epsilon_3$ and $\bar{M}_3$ are a pair of sub- and sup- solutions of $(\mathbf{P}_3)$. Hence
$$Q_{\epsilon_3}^{\bar{M}_3}:=[\epsilon_3, \bar{M}_3]=\big{\{}u\in \hat X_3: \epsilon_3\leq u(x)\leq \bar{M}_3 \big{\}}$$
is an invariant region of $(\mathbf{P}_3)$.
\end{itemize}
\end{remark}

\section{Persistence}\label{77}
The aim of this section is to investigate the uniform persistence of $(\textbf{P}_i)$ and to give the proof of Theorem \ref{73}.

\begin{proof}[Proof of Theorem \ref{73}]
 We only consider the case $i=1$ because the other cases can be dealt with similarly. We shall prove this theorem in four steps.
	
\textbf{Step 1.} Let $\underline{u}_0=0$ and $\bar{u}^0=u_1^*(x, t)$. In this step, we will prove that the system
	\begin{equation}\label{40}
	\left\{
	\begin{aligned}
	& \underline{u}_{kt}=\mathcal{L}_1[\underline{u}_{k}]+\underline{u}_{k}[a_1(x, t)-b_1(x, t)\underline{u}_{k}-c_1(x, t)G_1*\bar{u}^{k-1}],\\
	&\bar{u}^{k}_t=\mathcal{L}_1[\bar{u}^{k}]+\bar{u}^{k}[a_1(x, t)-b_1(x, t)\bar{u}^{k}-c_1(x, t)G_1*\underline{u}_{k}]
	\end{aligned}
	\right.
	\end{equation}
for $k=1,2,\ldots$, has exactly one positive periodic solution $\{(\underline{u}_k,\bar{u}^k)\}_{k=1}^{\infty}$, and that
 \begin{equation}\label{48}
 0<\underline{u}_1\leq \underline{u}_2\leq \cdots \leq \underline{u}_k\leq \cdots \leq \bar{u}^k\leq \cdots \leq \bar{u}^2\leq \bar{u}^1\leq \bar{u}^0.
 \end{equation}

	First, since $a_1(t, x)-c_1(x, t)G_1*u_1^*(x, t)>0$ and $\bar{u}^0=u_1^*(x, t)$, then it follows from Lemma \ref{basic-lm} that
	$$u_t=\mathcal{L}_1[u]+u[a_1(x, t)-c_1(x, t)G_1*\bar{u}_1^*-b_1(x, t)u]$$
	has exactly one positive periodic solution and we  denote it  by $\underline{u}_1$.
	It follows from the comparison principle that $\underline{u}_1\leq \bar{u}_1^*$ and
$$a_1(t, x)-c_1(x, t)G_1*\underline{u}_1\geq a_1(t, x)-c_1(x, t)G_1*\bar{u}_1^*>0.$$
Thus
	$$u_t=\mathcal{L}_1[u]+u[a_1(x, t)-c_1(x, t)G_1*\underline{u}_1-b_1(x, t)u]$$
	has exactly one positive periodic solution and we  denote it  by $\bar{u}^1$.
	Meanwhile, by comparison principle we also have $\bar{u}^1\geq \underline{u}_1$. Thus
	$$\underline{u}_0<\underline{u}_1\leq \bar{u}^1 \leq \bar{u}^0.$$
	
	Next, suppose that
for any $k=1,2,\cdots, m$,  $\underline{u}_k$ is a unique positive periodic solution of
\begin{equation*}
u_t=\mathcal{L}_1[u]+u[a_1(x, t)-c_1(x, t)G_1*\bar{u}^{k-1}(x, t)-b_1(x, t)u],
\end{equation*}
	and $\bar{u}^k$ is a unique positive periodic solution of
\begin{equation*}
u_t=\mathcal{L}_1[u]+u[a_1(x, t)-c_1(x, t)G_1*\underline{u}_{k}(x, t)-b_1(x, t)u],
\end{equation*}
	and that
	\begin{equation}\label{41}
\underline{u}_0<\underline{u}_1\leq \cdots \leq \underline{u}_m\leq \bar{u}^m\leq \cdots \leq \bar{u}^1 \leq \bar{u}^0.
\end{equation}
Note that
 $$a_1(t, x)-c_1(x, t)G_1*\bar{u}^{m}\geq a_1(t, x)-c_1(x, t)G_1*\bar{u}_1^*>0.$$
 This implies that equation
 $$u_t=\mathcal{L}_1[u]+u[a_1(x, t)-c_1(x, t)G_1*\bar{u}^{m}(x, t)-b_1(x, t)u]$$
 has exactly one positive periodic solution, which are denoted by $\underline{u}_{m+1}$. By the comparison principle, we obtain $\underline{u}_{m+1}\leq u_1^*$ and $$a_1(t, x)-c_1(x, t)G_1*\underline{u}_{m+1}\geq a_1(t, x)-c_1(x, t)G_1*\bar{u}_1^*>0.$$
 Then it implies that
	$$u_t=\mathcal{L}_1[u]+u[a_1(x, t)-c_1(x, t)G_1*\underline{u}_{m+1}(x, t)-b_1(x, t)u]$$
	has exactly one positive periodic solution, which is denoted by  $\bar{u}^{m+1}$. 
Using comparison principle and the fact that $\underline{u}_{m}\leq \bar{u}^m$, we have
$\underline{u}_{m+1}\leq \bar{u}^{m+1}$.
 It follows from  (\ref{41}) and the first equation of (\ref{40}) that $\underline{u}_m\leq \underline{u}_{m+1}$, which together with the second equation of (\ref{40}) implies
 $\bar{u}^{m+1}\leq \bar{u}^{m}$. Hence by induction, we can obtain two function sequences $\{\underline{u}_n\}$ and $\{\bar{u}^n\}$ satisfying \eqref{48}.

\medskip

 \textbf{Step 2}. In this step, we will prove that for any $u_0\in \hat{X}_1$ satisfying $\inf \limits_{x\in \mathbb{R}^N} u_0(x)>0$, any $n\in \mathbb{N}$, and any $0<\varepsilon<\bar{H}_1$ with
 \begin{equation}\label{65}
 \bar{H}_1=\min \limits_{(x, t)\in \R^N \times [0, T]} \underline{u}_1(x, t),
 \end{equation}
  there exists an increasing sequence $\{t^n_{\varepsilon, u_0}\}^{\infty}_{n=0}$ with $t^0_{\varepsilon, u_0}=0$ such that
 \begin{equation}\label{42}
 \underline{u}_n-\varepsilon\leq u_1(x, t; u_0)\leq \bar{u}^{n-1}+\varepsilon,
 \end{equation}
 for all $x \in \R^N$ and $t\geq t^n_{\varepsilon, u_0}$.

 To this end, we first claim that there exists $t^1_{\varepsilon, u_0}$ such that
 \begin{equation}\label{43}
 \underline{u}_1-\varepsilon\leq u_1(x, t; u_0)\leq \bar{u}^0+\varepsilon
 \end{equation}
  for all $x \in \R^N$ and $t\geq t^1_{\varepsilon, u_0}$. Note that
  \begin{equation}\label{60}
 u_t\leq \mathcal{L}_1[u]+u[a_1(x,t)-b_1(x,t)u],
  \end{equation}
  then by \cite[Theorem E]{Rawal2012},  we can find $\hat{t}^1_{\varepsilon, u_0}>0$ such that
   \begin{equation}\label{44}
   u_1(x, t; u_0)\leq \bar{u}^0+\varepsilon \quad \text{for all}\ t\geq \hat{t}^1_{\varepsilon, u_0}.
   \end{equation}
   Choose $0<\bar{\varepsilon}<\varepsilon$ such that
   $a_1(x,t)-c_1(x,t)G_1\ast \bar{u}^0-c_1(x,t)\bar{\varepsilon}>0$, then from Lemma \ref{basic-lm}, we know that the equation
 $$u_t=\mathcal{L}_1[u]+u[a_1(x,t)-b_1(x,t)u-c_1(x,t)G_1\ast \bar{u}^0-c_1(x,t)\bar{\varepsilon}]$$
admits exactly one positive periodic solution,  which is denoted by $\hat{u}(x, t, \bar{\varepsilon}, \bar{u}^0)$.
 From (\ref{44}), we know that for $t\geq \hat{t}^1_{\bar{\varepsilon}, u_0}$, $u_1(x,t;u_0)$ satisfies that
   \begin{equation*}
   u_t\geq \mathcal{L}_1[u]+u[a_1(x,t)-b_1(x,t)u-c_1(x,t)G_1\ast \bar{u}^0-c_1(x,t)\bar{\varepsilon}].
   \end{equation*}
   Thus, by Lemma \ref{basic-lm}, there exists $t^1_{\bar{\varepsilon}, u_0}\geq \hat{t}^1_{\bar{\varepsilon}, u_0}$ such that
   $$u_1(x, t; u_0)\geq \hat{u}(x, t, \bar{\varepsilon}, \bar{u}^0)-\bar{\varepsilon} \quad \text{for all}\ t\geq t^1_{\bar{\varepsilon}, u_0}.$$
   Since $\hat{u}(x, t, \bar{\varepsilon}, \bar{u}^0)\rightarrow \underline{u}_1$ as $\bar{\varepsilon}\rightarrow 0$ and $\hat{u}(x, t, \bar{\varepsilon}, \bar{u}^0)\leq \underline{u}_1$, we can find a small enough $\bar{\varepsilon}$ such that
   \begin{equation*}
   \hat{u}(x, t, \bar{\varepsilon}, \bar{u}^0)-\bar{\varepsilon}\geq \underline{u}_1-\varepsilon,
   \end{equation*}
   and hence for this $\bar{\varepsilon}$,
   \begin{equation}\label{45}
   u_1(x, t; u_0)\geq \underline{u}_1-\varepsilon \quad \text{for all}\ t\geq t^1_{\bar{\varepsilon}, u_0}.
   \end{equation}
   Let $t^1_{\varepsilon, u_0}=\max \{ \hat{t}^1_{\varepsilon, u_0}, t^1_{\bar{\varepsilon}, u_0}\}(>0)$, then (\ref{43}) holds for $t\geq t^1_{\varepsilon, u_0}$.

   Next, assume that for any $0<\varepsilon<\bar{H}_1$
   there exist $t^k_{\varepsilon, u_0}\geq t^{k-1}_{\varepsilon, u_0}\geq \cdots\geq  t^2_{\varepsilon, u_0}\geq t^{1}_{\varepsilon, u_0}$ with $k\geq 2$ such that
   \begin{equation}\label{46}
   \underline{u}_{k}-\varepsilon\leq u_1(x, t; u_0)\leq \bar{u}^{k-1}+\varepsilon
   \end{equation}
   for all $x\in \R^N$ and $t\geq t^k_{\varepsilon, u_0}$. We shall prove that there exists 
   $t^{k+1}_{\varepsilon, u_0}\geq t^{k}_{\varepsilon, u_0}$ such that
   \begin{equation}\label{47}
   \underline{u}_{k+1}-\varepsilon\leq u_1(x, t; u_0)\leq \bar{u}^{k}+\varepsilon
   \end{equation}
    for all $x\in \R^N$ and $t\geq t^{k+1}_{\varepsilon, u_0}$.
     Choose $0<\bar{\varepsilon}<\varepsilon$ such that
    $$a_1(x,t)-c_1(x,t)G_1\ast \bar{u}^0-c_1(x,t)\bar{\varepsilon}>0$$ 
    and hence that
    $$a_1(x,t)-c_1(x,t)G_1\ast \bar{u}^{k}-c_1(x,t)\bar{\varepsilon}\geq a_1(x,t)-c_1(x,t)G_1\ast \bar{u}^0-c_1(x,t)\bar{\varepsilon}>0.$$
Then by Lemma \ref{basic-lm}, each of the following two equations
$$u_t=\mathcal{L}_1[u]+u[a_1(x,t)-b_1(x,t)u-c_1(x,t)G_1\ast \underline{u}_k+c_1(x,t)\bar{\varepsilon}]$$
and
$$u_t=\mathcal{L}_1[u]+u[a_1(x,t)-b_1(x,t)u-c_1(x,t)G_1\ast \bar{u}^{k}-c_1(x,t)\bar{\varepsilon}] $$
admits exactly one positive periodic solution, which are denoted by $\hat{u}(x, t, \bar{\varepsilon}, \underline{u}_k)$ and $\hat{u}(x, t, \bar{\varepsilon}, \bar{u}^{k})$, respectively.
It follows from (\ref{46}) that for $t\geq t^k_{\bar{\varepsilon}, u_0}$, $u_1(x,t;u_0)$ satisfies
    \begin{equation}\label{49}
  u_t\leq \mathcal{L}_1[u]+u[a_1(x,t)-b_1(x,t)u-c_1(x,t)G_1\ast \underline{u}_k+c_1(x,t)\bar{\varepsilon}].
    \end{equation}
Thus, there exists $t^k_{\bar{\varepsilon}, u_0, \underline{u}_k}\geq t^k_{\bar{\varepsilon}, u_0}$  such that
\begin{equation*}\label{51}
u_1(x, t; u_0)\leq \hat{u}(x, t, \bar{\varepsilon}, \underline{u}_k)+\bar{\varepsilon} \quad \text{for all}\ t\geq t^k_{\bar{\varepsilon}, u_0, \underline{u}_k}.
\end{equation*}
Since $\hat{u}(x, t, \bar{\varepsilon}, \underline{u}_k)\rightarrow \bar{u}^k$ as $\bar{\varepsilon}\rightarrow 0$ and $\hat{u}(x, t, \bar{\varepsilon}, \underline{u}_k)\geq \bar{u}^k$, we can find a small enough $\bar{\varepsilon}$ such that
$\hat{u}(x, t, \bar{\varepsilon}, \underline{u}_k)+\bar{\varepsilon}\leq \bar{u}^k+\varepsilon$,
which means
\begin{equation}\label{53}
u_1(x, t;u_0)\leq \bar{u}^k+\varepsilon \quad \text{for all}\ t\geq t^k_{\bar{\varepsilon}, u_0, \underline{u}_k}.
\end{equation}
It follows from (\ref{53}) that  there exists $ t^k_{\bar{\varepsilon}, u_0, \bar{u}^k}\geq  t^k_{\bar{\varepsilon}, u_0, \underline{u}_k}$ such that for $t\geq t^k_{\bar{\varepsilon}, u_0, \bar{u}^k}$, $u_1(x,t;u_0)$ satisfies
\begin{equation}\label{54}
u_t\geq \mathcal{L}_1[u]+u[a_1(x,t)-b_1(x,t)u-c_1(x,t)G_1\ast \bar{u}^{k}-c_1(x,t)\bar{\varepsilon}].
\end{equation}
Thus from (\ref{54}) and by Lemma \ref{basic-lm}, we can find $\hat{t}^k_{\bar{\varepsilon}, u_0, \bar{u}^k}\geq t^k_{\bar{\varepsilon}, u_0, \bar{u}^k}$ such that
\begin{equation*}\label{55}
u_1(x, t; u_0)\geq \hat{u}(x, t, \bar{\varepsilon}, \bar{u}^{k})-\bar{\varepsilon} \quad \text{for all}\ t\geq \hat{t}^k_{\bar{\varepsilon}, u_0, \bar{u}^k}.
\end{equation*}
Since $\hat{u}(x, t, \bar{\varepsilon}, \bar{u}^{k})\rightarrow \underline{u}_{k+1}$ as $\bar{\varepsilon}\rightarrow 0$ and $\hat{u}(x, t, \bar{\varepsilon}, \bar{u}^{k})\leq \underline{u}_{k+1}$, there exists a small enough $\bar{\varepsilon}$ such that
$$\hat{u}(x, t, \bar{\varepsilon}, \bar{u}^{k})-\bar{\varepsilon}\geq \underline{u}_{k+1}-\varepsilon,$$
which means
\begin{equation}\label{56}
u_1(x, t; u_0)\geq \underline{u}_{k+1}-\varepsilon \quad \text{for all}\ t\geq \hat{t}^k_{\bar{\varepsilon}, u_0, \bar{u}^k}.
\end{equation}
Let $t^{k+1}_{\varepsilon, u_0}=\max\{ t^k_{\bar{\varepsilon}, u_0, \underline{u}_k},\ \hat{t}^k_{\bar{\varepsilon}, u_0, \bar{u}^k}\}(\geq t^k_{\varepsilon, u_0})$,  then (\ref{47}) holds for $t\geq t^{k+1}_{\varepsilon, u_0}$. Therefore by induction, (\ref{42}) holds for any $n\geq 1$.

\medskip

\textbf{Step 3.} In this step, we prove that for any given $n\in \mathbb{N}$ and $n\geq 1$, $u_0\in \hat{X}_i$ with $\inf \limits_{x\in \mathbb{R}^N} u_0(x)>0$, if
\begin{equation}\label{57}
\underline{u}_n\leq u_0\leq \bar{u}^{n-1},
\end{equation}
then
\begin{equation}\label{58}
\underline{u}_n\leq u_1(x, t; u_0)\leq \bar{u}^{n-1} \quad \text{for all} \  t\geq 0
\end{equation}

 To this end,  we first prove
\begin{equation}\label{59}
\underline{u}_1\leq u_1(x, t; u_0)\leq \bar{u}^0 \quad \text{for all} \  t\geq 0.
\end{equation}
It follows from (\ref{60}) and $u_0\leq \bar{u}^{n-1}\leq \bar{u}^0$ that
\begin{equation}\label{61}
u_1(x, t; u_0)\leq \bar{u}^0, \quad t\geq 0.
\end{equation}
Then $u_1(x,t;u_0)$ satisfies
$$u_t\geq \mathcal{L}_1[u]+u[a_1(x, t)-c_1(x, t)G_1*\bar{u}^0-b_1(x, t)u].$$
 for $t\ge 0$. By comparison principle and $\underline{u}_1\leq \underline{u}_n\leq u_0$, we have $\underline{u}_1\leq u_1(x, t; u_0)$ for $t\geq 0$, which together with (\ref{61}) implies (\ref{59}).

Next, suppose that $\underline{u}_k\leq u_1(x, t; u_0)\leq \bar{u}^{k-1}$ for $k=2,3,\cdots, j$  $(j\leq n-1)$. We shall prove that
\begin{equation}\label{62}
\underline{u}_{j+1}\leq u_1(x, t; t_0,u_0)\leq \bar{u}^{j}.
\end{equation}
Since $u_1(x,t;u_0)$ satisfies
 $$u_t\leq \mathcal{L}_1[u]+u[a_1(x, t)-c_1(x, t)G_1*\underline{u}_j(x, t)-b_1(x, t)u]$$
   for $t\ge 0$, and $u_0\leq \bar{u}^{n-1}\leq \bar{u}^{j}$, we obtain
\begin{equation}\label{63}
u_1(x, t; u_0)\leq \bar{u}^j \quad t\geq 0.
\end{equation}
Meanwhile, we have $\underline{u}_{j+1}\leq \underline{u}_{n}\leq u_0$ and $u_1(x,t;u_0)$satisfies
$$u_t\geq \mathcal{L}_1[u]+u[a_1(x, t)-c_1(x, t)G_1*\bar{u}^j(x, t)-b_1(x, t)u],$$
 for $t\ge 0$, which implies
\begin{equation}\label{64}
u_1(x, t; u_0)\geq \underline{u}_{j+1}, \quad t\geq 0.
\end{equation}
Thus (\ref{63}) and (\ref{64}) imply (\ref{62}). By induction, we get (\ref{58}).

\medskip

\textbf{Step 4.} From Step 1, we know that $\{\underline{u}_n\}$ is  a non-decreasing bounded function sequence and $\{\bar{u}^n\}$ is a non-increasing bounded function sequence. Thus $\lim \limits_{n\rightarrow \infty} \underline{u}_n$ and $\lim \limits_{n\rightarrow \infty} \bar{u}^n$ exist. Let
\begin{equation}\label{71}
 \underline{U}^1(x,t)=\lim \limits_{n\rightarrow \infty} \underline{u}_n(x,t), \quad \overline{U}^1(x,t)=\lim \limits_{n\rightarrow \infty} \bar{u}^n(x,t),
\end{equation}
then $\underline{U}^1$ and $\overline{U}^1$ are bounded positive periodic functions. In this step, we shall prove that $\underline{U}^1$ and $\bar{U}^1$
satisfy the properties stated in Theorem \ref{73}.

First, it follows from (\ref{40}) that
	\begin{align}\label{68}
	\bar{u}^n(x, t)&=\bar{u}^n(x, 0)+\int_0^{t} { \int_{\R^N} J(y-x)(y, s)dy}+
\bar{u}^n(x, s) \left({ -1+a_1(x, s)} \right. \nonumber \\
&\qquad \left.-b_1(x, s)\bar{u}^n(x, s)-c_1(x, s)G_1*\underline{u}_n(x, s)\right) ds.
	\end{align}
Letting $n\rightarrow \infty$ in (\ref{68}) and using the Dominated Convergence Theorem, we obtain
	 \begin{align}
\label{new-add-eq1}
\overline{U}^1(x, t)&=\overline{U}^1(x, 0)+\int_0^{t}{ \int_{\R^N}J(y-x)(y, s)dy}+
\overline{U}^1(x, s)\left ({ -1+a_1(x, s)}\right. \nonumber \\
&\qquad \left.-b_1(x, s)\overline {U}^1(x, s)-c_1(x, s)G_1*\underline{U}^1(x, s)\right)ds.
\end{align}
{ Similarly,
 \begin{align}
\label{new-add-eq1-1}
\underline{U}^1(x, t)&=\underline{U}^1(x, 0)+\int_0^{t}{ \int_{\R^N}J(y-x)(y, s)dy}+
\underline{U}^1(x, s)\left ({ -1+a_1(x, s)}\right. \nonumber \\
&\qquad \left.-b_1(x, s)\underline {U}^1(x, s)-c_1(x, s)G_1*\overline{U}^1(x, s)\right)ds.
\end{align}
Then by the arguments similar to those in the proof of the continuity and differentiability of $\underline{u}_1(x,t)$ and $\overline{u}_1(x,t)$ in $t$ in
Theorem \ref{98} $(2)$,  we have that $\overline{U}^1(x, t)$ and $\underline{U}_1(x,t)$  are continuous as well as differentiable  in $t$.
 Moreover, we have
\begin{equation}\label{69}
\begin{cases}
\overline{U}^1_t(x,t)=\mathcal{L}_1[\overline{U}^1](x,t)+\overline{U}^1(x,t)[a_1(x,t)-b_1(x,t)\overline{U}^1(x,t)-c_1(x,t)\big(G_1\ast \underline{U}^1\big)(x,t)]\cr
\underline{U}^1_t(x,t)=\mathcal{L}_1[\underline{U}^1](x,t)+\underline{U}^1(x,t)[a_1(x,t)-b_1(x,t)\underline{U}^1(x,t)-c_1(x,t)\big(G_1\ast \overline{U}^1\big)(x,t)]
\end{cases}
\end{equation}
for  $t\in\R$.}

Let $\tilde{a}_1(x,t)=a_1(x,t)-c_1(x,t)\big(G_1\ast \underline{U}^1\big)(x,t)(>0)$. { Then $\tilde{a}_1(x,t)$ is continuous in $x\in\R^N$ and $t\in\R$,} and $\overline{U}^1$ is a positive periodic solution of
\begin{equation}\label{70}
u_t=\mathcal{L}_1[u]+u[\tilde{a}_1(x,t)-b_1(x, t)u].
\end{equation}
 By \cite[Theorem E]{Rawal2012}, (\ref{70}) has exactly one positive periodic solution, which is continuous in $x$ and is asymptotically stable with respect to strictly positive perturbations. We denote it by $\bar U^*$. By \cite[Theorem E]{Rawal2012} and \cite[Proposition 3.1]{shen2012traveling}, we have $\bar U^1=\bar U^*$ and thus
   $\overline{U}^1$ is continuous in $x$. Similarly, $\underline{U}_1$ is also continuous in $x$.

Next we prove (\ref{66}). { By the continuity and periodicity of $\overline{U}_1(x,t)$ and $\underline{U}_1(x,t)$ in $x\in\R^N$ and $t\in\R$, the convergence in \eqref{71} is uniform with respect to $x\in\R^N$ and $t\in\R$.} It then  follows from (\ref{42}) and (\ref{71}) that for any $\varepsilon>0$, there exists $N$ such that
$$0<\underline{U}^1-2\varepsilon\leq \underline{u}_N-\varepsilon\leq u_1(x, t; u_0)\leq \bar{u}_{N-1}+\varepsilon\leq \overline{U}^1+\varepsilon$$
for all $x \in B_1$ and $t\geq t^N_{\varepsilon, u_0}$ which implies that (\ref{66}) holds.

Now if $\underline{U}^1\leq u_0\leq \overline{U}^1$, then it follows from (\ref{48}) that
$$\underline{u}_n\leq \underline{U}^1\leq u_0\leq \overline{U}^1\leq \bar{u}^{n-1}$$
 for $n\in \mathbb{N}$.
By the conclusion in Step 2, we have
$$\underline{u}_n\leq u_1(x, t; u_0)\leq \bar{u}^{n-1}, \quad \text{for all}\ n\in \mathbb{N}, \,\, t\geq 0.$$
Letting $n\rightarrow \infty$, we obtain
$$\underline{U}^1\leq u_1(x, t; u_0)\leq \overline{U}^1.$$
 Theorem \ref{73} is thus proved.
\end{proof}

\begin{remark}
\begin{itemize}
	\item [$(1)$]
 Since for $i=1,2,3$,  $\underline{U}^i$ and $\overline{U}^i$ satisfy $\overline{U}^i\geq \underline{U}^i$ and
 \begin{equation*}
 \left\{
 \begin{aligned}
 & \underline{U}^i_{t}=\mathcal{L}_i[\underline{U}^i]+\underline{U}^i[a_i(x, t)-b_i(x, t)\underline{U}^i-c_i(x, t)G_i*\overline{U}^i],\\
 &\overline{U}^i_t=\mathcal{L}_i[\overline{U}^i]+\overline{U}^i[a_i(x, t)-b_i(x, t)\overline{U}^i-c_i(x, t)G_i*\underline{U}^i],
 \end{aligned}
 \right.
 \end{equation*}
then  $\underline{U}^i$ and $\overline{U}^i$ can be viewed as a pair of  sub-solution and sup-solution of $(\mathbf{P}_i)$. Therefore the last part of Theorem \ref{73} also can be obtained by Theorem \ref{98}(2).

 \item[$(2)$] Similar arguments as those in the proof of Theorem  \ref{73}  are first established in the proof of \cite[Theorem 1.2]{SaSh}.
     \end{itemize}	
\end{remark}

\section{Existence, uniqueness, and stability of positive time periodic solutions}\label{76}

 In this section, we study  the existence, uniqueness  and stability of positive time periodic solutions for $(\mathbf{P}_1)$-$(\mathbf{P}_3)$,
 and prove Theorem \ref{104}.

 \subsection{Proof of  Theorem \ref{104} (1)}

 In this subsection, we prove Theorem \ref{104} (1).

  First, we give a definition of sup- and sub- solutions of $(\mathbf{P}_i)$, $i=1,2,3$ in a standard way.
  \begin{definition}\label{30}
	For $i=1,2,3$, a non-negative bounded continuous function $u(x, t)$ on ${ \Omega_i}\times [0,  \infty)$ is said to be a sup-solution (or sub-solution) of $(\mathbf{P}_i)$ if
	$\frac{\partial u}{\partial t}$ exists and is continuous on ${ \Omega_i}\times [0,  \infty)$,
and
  \begin{equation*}
		\begin{aligned}
		& u_t(x,t)-\mathcal{L}_i[u](x,t)\geq (\leq) u(x,t)[a_i(x,t)-b_i(x,t)u(x,t)-c_i(x,t)\big(G_i\ast u\big)(x,t)]
		\end{aligned}
		\end{equation*}
		for $(x, t)\in { \Omega_i}\times [0, \infty)$.
\end{definition}

For $i=1,2,3$, let
   $$\check{C}_i:=\{C: J(x)\geq c_{iM}CG(x), \ { x\in B_{r_1}}\}, \quad C_{i0}=\max \check{C}_i.$$
Under the assumptions $(\textbf{A}_0)$ and $(\textbf{A}_3)$, $C_{i0}<+\infty$, $\check{C}_i\neq \emptyset$, and $C_{i0}> \frac{a_{iM}}{b_{iL}}$.

\begin{lemma}\label{31}
	 Assume that assumption ($\textbf{A}_3$) holds. For $i=1,2,3$, if $u_1$ and $u_2$ are  sup- and sub- solutions of $(\mathbf{P}_i)$ 
on ${ \Omega_i}\times [0, \infty)$, respectively, and satisfy
	$$ u_1(x, 0)\geq u_2(x, 0) \quad x\in { \Omega_i},$$
{ and
$$
u_2\le C_{i0}\quad {\rm or}\quad u_1\le C_{i0},
$$
then}  $u_1(x, t)\geq u_2(x, t)$ in ${ \Omega_i}\times (0, \infty)$.
\end{lemma}

\begin{proof}
 We only consider the case where $i=1$ because the other cases can be dealt with similarly.
In view of Definition \ref{30}, { there is $M_0>0$ such that}
	$$0\leq u_i(x, t) { \le  M_0} \quad\mbox{for}\quad \quad (x, t)\in \mathbb{R}^N\times [0, \infty), \,\, i=1,2.$$

Let $W=e^{ct}(u_1-u_2)$.   Then
	\begin{align}\label{4}
	W _t & = ce^{ct}(u_1-u_2)+e^{ct}( u_{1t}-u_{2t}) \nonumber \\
	&\geq \int_{\mathbb{R}^N}J(y-x)W(y, t)dy+\left [c+a_1(x, t)-1-b_1(x, t)(u_1+u_2) \right.\nonumber\\
	&\qquad \left.  -c_1(x, t)G_1*u_1\right]W-c_1(x, t)u_2G_1*W.
	\end{align}
and
	\begin{align}\label{new-add-eq2}
	W _t & = ce^{ct}(u_1-u_2)+e^{ct}( u_{1t}-u_{2t}) \nonumber \\
	&\geq \int_{\mathbb{R}^N}J(y-x)W(y, t)dy+\left [c+a_1(x, t)-1-b_1(x, t)(u_1+u_2) \right.\nonumber\\
	&\qquad \left.  -c_1(x, t)G_1*u_2\right]W-c_1(x, t)u_1G_1*W.
	\end{align}
We can choose $c$ large enough such that for any $(x, t)\in \mathbb{R}^N\times [0, \infty)$,	
	\begin{equation*}
	c+a_1(x, t)-1-b_1(x, t)(u_1+u_2)-c_1(x, t)G_1*u_1\geq a_{1L}>0.
	\end{equation*}
and 	
	\begin{equation*}
	c+a_1(x, t)-1-b_1(x, t)(u_1+u_2)-c_1(x, t)G_1*u_2\geq a_{1L}>0.
	\end{equation*}
Then it follows from the first part of Lemma \ref{116} $(1)$ and assumption ($\textbf{A}_3$) that $u_2(x, t)\leq u_1(x, t)$ for $(x, t)\in \mathbb{R}^N\times (0, \infty)$.

\end{proof}

\begin{corollary}\label{32}
 Assume that assumption ($\textbf{A}_3$) holds. For $i=1,2,3$, let $u_1\leq C_{i0}$ and $u_2\leq C_{i0}$ be  sup- and sub- solutions of $(\mathbf{P}_i)$ 
 on ${ \Omega_i}\times [0, \infty)$, respectively, and $u_1(\cdot, t),\ u_2(\cdot, t)\in \hat{X}_i$ for all $t\geq0$. Then for any $u_0\in \hat{X}_i$ satisfying
	$$u_1(x, 0)\geq u_0(x)\geq u_2(x, 0),$$
	equation $(\mathbf{P}_i)$ admits a solution $u_i(x, t; u_0)$ on ${ \Omega_i}\times [0,\infty)$ which satisfies
	$$u_2(x, t)\leq u_i(x, t; u_0)\leq u_1(x, t)\quad \mbox{for} \quad (x, t)\in { \Omega_i}\times [0,\infty).$$
\end{corollary}

\begin{proof}
{ It follows from Lemma \ref{31} directly.}
\end{proof}


Furthermore,  we have the following comparison principle for $(\mathbf{P}_i)$.
\begin{corollary}\label{16-cor}
Assume that assumption ($\textbf{A}_3$)  holds. For $i=1,2,3$,	let $u_1(x, t)$ and $u_2(x, t)$ be solutions of $(\mathbf{P}_i)$ with initial value $u_{10}\in \hat{X}_i$ and $u_{20}\in \hat{X}_i$, respectively. If the initial values  $u_{20}$, $ u_{10}$ satisfy
$$u_{20}\not \equiv u_{10}, \quad   0\leq u_{20}\leq u_{10}$$
and
$$ u_{10} \leq \frac{a_{iM}}{b_{iL}}\quad \text{or} \quad  u_{20} \leq \frac{a_{iM}}{b_{iL}}.$$
	then	
$u_2(x, t)<u_1(x, t)$ for $(x, t)\in { \Omega_i}\times (0, \infty)$.
\end{corollary}

\begin{proof}
 Note that $u(x,t)\equiv \frac{a_{iM}}{b_{iL}}$ is a super-solution of ($\textbf{P}_i$) and $u\equiv 0$ is a sub-solution of ($\textbf{P}_i$), and that $\frac{a_{iM}}{b_{iL}}< C_{i0}$. Then by Corollary \ref{32}, we have
$$
0\le u_2(x,t)\le \frac{a_{iM}}{b_{iL}}.
$$
By Lemma \ref{32}, we have
$$
u_2(x,t)\le u_1(x,t).
$$
 Let
  $W_1=e^{ct}(u_1-u_2)$, then
  \begin{align*}
  W_{1t}&=\int_{\mathbb{R}^N}J(y-x)W_1(y, t)dy+\left [c+a_1(x, t)-1-b_1(x, t)(u_1+u_2)\right. \nonumber \\
 	&\qquad \left. -c_1(x, t)G_1*u_1\right]W_1-c_1(x, t)u_2G_1*W_1.
  \end{align*}
  and
    \begin{align*}
  W_{1t}&=\int_{\mathbb{R}^N}J(y-x)W_1(y, t)dy+\left [c+a_1(x, t)-1-b_1(x, t)(u_1+u_2)\right. \nonumber \\
 	&\qquad \left. -c_1(x, t)G_1*u_2\right]W_1-c_1(x, t)u_1G_1*W_1.
  \end{align*}
  We can choose $c$ large enough such that for any $(x, t)\in \mathbb{R}^N\times [0, \infty)$,	
	\begin{equation*}
	c+a_1(x, t)-1-b_1(x, t)(u_1+u_2)-c_1(x, t)G_1*u_1\geq a_{1L}>0.
	\end{equation*}
and 	
	\begin{equation*}
	c+a_1(x, t)-1-b_1(x, t)(u_1+u_2)-c_1(x, t)G_1*u_2\geq a_{1L}>0.
	\end{equation*}
Then it follows from the second part of Lemma \ref{116} $(1)$ and assumption ($\textbf{A}_3$) that	
$u_2(x, t)<u_1(x, t)$ for $(x, t)\in { \Omega_i}\times (0, \infty)$
\end{proof}
We now prove Theorem \ref{104} (1).
\begin{proof}[Proof of Theorem \ref{104} (1)]

 We only consider the case where $i=1$ because other cases can be dealt with similarly. The proof can be divided into the following four steps.	

\textbf{Step 1}. By the assumption ($\textbf{A}_3$) and Definition \ref{30}, we  know that $\bar{u}=C_{10}$ and $\underline{u}=\epsilon_0$ are  sup-solution and sub-solution of $(\mathbf{P}_1)$, respectively, where $\epsilon_0>0$ is a small enough constant. Recall that  $u_1(x, t; u_0)$ is the solution of $(\mathbf{P}_1)$ with initial value $u_0$,
For  any $(x, t)\in R^N \times [0,\infty)$, we denote
$$\bar{u}_n(x, t)=u_1(x, t+nT;\bar{u}), \quad \underline{u}_n(x,t)=u_1(x, t+nT;\underline{u}).$$
By{ Corollary \ref{32}}, we know that 
$\{\bar{u}_n\}_{n=1}^{\infty} $ is a monotone decreasing sequence  and $\{\underline{u}_n\}_{n=1}^{\infty} $ is a monotone increasing sequence with respect to $n$; meanwhile, $\bar{u}_n \geq \underline{u}$ and $\underline{u}_n \leq \bar{u}$ for all $n$. Thus, we can define the following two functions:
$$u^+(x, t)=\lim_{n\rightarrow + \infty} u_1(x, t+nT; \bar{u}),\quad u^-(x, t)=\lim_{n\rightarrow \infty} u_1(x, t+nT; \underline{u}).$$
Then $u^+(x, t)\geq u^-(x, t)$,  and $u^+(x,t)$ and $u^-(x,t)$ are periodic in both $x$ and $t$.

 \textbf{Step 2}. In this step, we will claim that for any $x\in \R^N$, $u^+(x, t)$ and $u^-(x, t)$ are uniformly continuous with respect to $t$ and that for any $t\geq 0$, $u^+(x, t)$ and $u^-(x, t)$ are uniformly continuous with respect to $x$. In what follows, we only discuss $u^+(x, t)$, and the case for {$u^-(x, t)$} can be analyzed similarly.

 Notice that
{ \begin{align*}
\bar{u}_n(x, t)&=\bar{u}_n(x, 0)+\int_0^{t} \big{[}\mathcal{L}_1[\bar{u}_n(x, s)]+
\bar{u}_n(x, s)\left (a_1(x, s)\right.\\
&\qquad \left.-b_1(x, s)\bar{u}_n(x, s)-c_1(x, s)G_1*\bar{u}_n(x, s)\right) \big{]}ds.
\end{align*}}
Letting $n\rightarrow \infty$ and by dominated convergence theorem, we obtain
\begin{align*}
u^+(x, t)&=u^+(x, 0)+\int_0^{t} \big{[}\mathcal{L}_1[u^+(x, s)]+u^+(x, s) \left (a_1(x, s)\right. \\
&\qquad \left.-b_1(x, s)u^+(x, s)-c_1(x, s)G_1*u^+(x, s)\right) \big{]}ds.	
\end{align*}
{ By the arguments similar to those in the proof of the continuity and differentiability of $\underline{u}_1(x,t)$ and $\overline{u}_1(x,t)$ in $t$ in
Theorem \ref{98} $(2)$, we have that $u^+(x, t)$ is continuous as well as differentiable in $t$, and}
\begin{equation}\label{33}
u^+_t=J \ast u^+-u^++u^+[a_1(x,t)-b_1(x,t)u^+-c_1(x,t)G_1\ast u^+]
\end{equation}
for $t\in \R$.

For each fixed $x$, set
$$a_x(t)=a_1(x, t), \ b_x(t)=b_1(x, t),\ c_x(t)=c_1(x, t),\ u^+_x(t)=u^+(x, t),$$
and consider the following
auxiliary equation
  \begin{equation}\label{34}
u_t=u[a_x(t)-1-b_x(t)u-c_x(t)(G_1\ast u^+)(x, t)]+\int_{\mathbb{R}^N}J(y-x)u^+(y, t)dy,
\end{equation}
{Obviously, $u^+_x$ is a positive periodic solution of (\ref{34}). }
We claim that (\ref{34}) has exactly one positive time periodic solution. Indeed, If not, then (\ref{34}) has two positive time periodic solutions $\hat{u}_1$ and $\hat{u}_2$ satisfying $\hat{u}_1(t)>\hat{u}_2(t)$ for $t\in [0,\ T]$. Let $\hat{u}=\hat{u}_1-\hat{u}_2$, then
	\begin{align*}
	\hat{u}_t & =\hat{u}_{1t}-\hat{u}_{2t} \\
	&= [a_x(t)-1-b_x(t)\hat{u}_1-c_x(t)G_1\ast u^+]\hat{u}_1+\int_{\mathbb{R}^N}J(y-x)u^+(y, t)dy \\
	&\quad
	-[a_x(t)-1-b_x(t)\hat{u}_2-c_x(t)G_1\ast u^+]\hat{u}_2-\int_{\mathbb{R}^N}J(y-x)u^+(y, t)dy
	\\
	&=\hat{u}[a_x(t)-1-b_x(t)(\hat{u}_1+\hat{u}_2)-c_x(t)G_1\ast u^+]\\
	&<\hat{u}[a_x(t)-1-b_x(t)\hat{u}_1-c_x(t)G_1\ast u^+].
	\end{align*}
Thus, we have
\begin{align*}
\hat{u}(T)&<\hat{u}(0)\exp\big{(}\int_0^T (a_x(s)-1-b_x(s)\hat{u}_1(s)-c_x(s)(G_1\ast u^+)(x, s))ds\big{)} \\
&=\hat{u}(0)\exp\big{(}\int_0^T (\frac{\hat{u}_{1t}-\int_{\mathbb{R}^N}J(y-x)u^+(y, s)dy}{\hat{u}_1})ds\big{)}\\
&<\hat{u}(0)\exp(\int_0^T \frac{\hat{u}_{1t}}{\hat{u}_1} ds)= \hat{u}(0),
\end{align*}
which is a contradiction. Therefore, we have that for any $x$, $u_x^+$ is the unique positive periodic solution of (\ref{34}).

 Given a sequence $\{x_n\}$ and $x$ satisfying $x_n\rightarrow x$ as $n$ goes to $\infty$, we denote $u^+_n(t)=u^+(x_n, t)$. It is easy to see that $u^+_n(t)$ is uniformly bounded and equi-continuous on $[0, T]$. By Arzela Ascoli theorem, there exists a subsequence, still denoted by $\{u^+_n\}$, which converges uniformly to a certain periodic function $u^*$. Recall that
\begin{equation}\label{36}
u^+_{nt}=u^+_n[a_{x_n}(t)-1-b_{x_n}(t)u^+_n-c_{x_n}(t)(G\ast u^+)(x_n, t)]+\int_{\mathbb{R}^N}J(y-x_n)u^+(y, t)dy.
\end{equation}
Letting $\ n\rightarrow \infty$, we get
$$u^+_{nt}\rightarrow u^*[a_x(t)-1-b_x(t)u^*-c_x(t)(G\ast u^+)(x, t)]+\int_{\mathbb{R}^N}J(y-x)u^+(y, t)dy$$
uniformly on $[0, T]$. This implies that $u^*$ is differentiable in $t$ and
$$u^*_{t}=u^*[a_x(t)-1-b_x(t)u^*-c_x(t)(G\ast u^+)(x, t)]+\int_{\mathbb{R}^N}J(y-x)u^+(y, t)dy.$$
Thus, we get $u^*(t)=u^+(x, t)$, which implies the continuity of $u^+(x, t)$ in $x$.

 \textbf{Step 3}. Consider  the following two equations
\begin{equation}\label{37}
u_t=\int_{\mathbb{R}^N}J(y-x)u(y, t)dy-u+up^+(x, t)
\end{equation}
and
\begin{equation}\label{38}
u_t=\int_{\mathbb{R}^N}J(y-x)u(y, t)dy-u+up^-(x, t),
\end{equation}
where
 \begin{align*}
   &p^+(x, t)=a_1(x, t)-b_1(x, t)u^+-c_1(x, t)G_1\ast u^+,\\
   & p^-(x, t)=a_1(x, t)-b_1(x, t)u^--c_1(x, t)G_1\ast u^-.
 \end{align*}
 If $u^+\not \equiv u^-$, by Corollary \ref{16-cor} and the periodicity of $u^+$ and $u^-$, we get $u^+(x, t)>u^-(x, t)$ for $(x, t)\in \mathbb{R}^N \times [0, \ \infty)$. Thus, $p^+<p^-$.
We can find $c^*>0$ such that $p^+\leq p^--c^*$. Note that { for any $c_*>0$},
\begin{align*}
	(c_*u^-e^{-c^*t})_t & =c_*u^-_te^{-c^*t}-c_*c^*u^-e^{-c^*t} \\
	&= \int_{\mathbb{R}^N}J(y-x)(c_*u^-e^{-c^*t})(y, t)dy-c_*u^-e^{-c^*t}+c_*u^-e^{-c^*t}(p^--c^*)
	\\
	&\geq \int_{\mathbb{R}^N}J(y-x)(c_*u^-e^{-c^*t})(y, t)dy-c_*u^-e^{-c^*t}+c_*u^-e^{-c^*t}p^+.
	\end{align*}
{Hence, $c_*u^-e^{-c^*t}$ is a sup-solution of (\ref{37})}. Since $\underline{u}\leq u^+ \leq \bar{u}$, we can choose $c_*$ such that $c_*u^-(x, 0)\geq u^+(x, 0)$. Then, by Proposition \ref{100}, we have $u^+(x, t)\leq c_*e^{-c^*t}u^-(x, t)$ for $(x, t)\in \mathbb{R}^N \times [0, \ \infty)$,  which is a contradiction by the boundedness and periodicity of $u^-$ and $u^+$.
Therefore, $u^+\equiv u^-=:u_P$. From the above discussion and Corollary \ref{32}, for any initial condition $u_0$ satisfying $\inf \limits_{x\in \mathbb{R}^N} u_0(x)>0$ and $0<u_0\leq \bar{u}$, we have
\begin{equation}\label{39}
\|u_1(\cdot, t; u_0)-u_P(\cdot, t)\|_{\hat{X}_1}\rightarrow 0, \quad  t\rightarrow \infty.
\end{equation}

For any $u_0$ satisfying $\inf \limits_{x\in \mathbb{R}^N} u_0(x)>0$, we know that $u_1(x, t; u_0)\leq \hat{u}_1(x,t;u_0)$, where $\hat{u}_1(x,t;u_0)$ is defined in Section \ref{102}. On the other hand, $$\limsup\limits_{t\rightarrow \infty}\hat{u}_1(x,t;u_0) \leq \frac{a_{1M}}{b_{1L}}.$$
Then for any $\varepsilon>0$ small enough, we can find $T_0$ such that $$u_1(x, T_0; u_0)\leq\sup\limits_{x\in D_1}\hat{u}_1(x,T_0;u_0)\leq \frac{a_{1M}}{b_{1L}}+\varepsilon\leq C_{10},$$
which implies that (\ref{39}) also holds.
\end{proof}

\subsection{Proof of Theorem \ref{104} (2) }

In this subsection, we prove Theorem \ref{104} (2).

\begin{proof} [Proof of Theorem \ref{104} (2)]
We only consider the case $i=1$ and the other cases can be dealt with similarly.
 By Theorem \ref{73},
 we can obtain two positive time periodic functions $\underline{U}$ and $\overline{U}$ which satisfy $\underline{U}\leq \overline{U}$ and
 \begin{equation*}
 \left\{
 \begin{aligned}
 & \underline{U}_{t}=\mathcal{L}_1[\underline{U}]+\underline{U}[a_1(x, t)-b_1(x, t)\underline{U}-c_1(x, t)G_1*\overline{U}],\\
 &\overline{U}_t=\mathcal{L}_1[\overline{U}]+\overline{U}[a_1(x, t)-b_1(x, t)\overline{U}-c_1(x, t)G_1*\underline{U}].
 \end{aligned}
 \right.
 \end{equation*}
 Moreover, by Theorem \ref{73}, if $u_0\in X^+_1\setminus \{0\}$, then 
$\overline{U}$, $ \underline{U}\in X^{++}_1$
and
 for any small enough $\epsilon>0$, there is $t_{\epsilon,u_0}>0$ such that 
 $$
 \underline{U}(x,t)-\epsilon\le u_1(t,x;u_0)\le \bar U(x,t)+\epsilon\quad \forall\,\, t\ge t_{\epsilon,u_0},\,\, x\in \R^N.
 $$
 It then suffices to prove that $\underline{U}(x,t)\equiv \overline{U}(x,t)$.

To do so, let $W=\overline{U}-\underline{U}$, then
\begin{align}\label{78}
	W_t & =\overline{U}_{t}-\underline{U}_{t} \nonumber \\
	&=\mathcal{L}_1[\overline{U}]+a_1(x, t)\overline{U}-b_1(x, t)\overline{U}^2-c_1(x, t)\overline{U}G_1*\underline{U} \nonumber \\
	&\quad
	-\mathcal{L}_1[\underline{U}]-a_1(x, t)\underline{U}+b_1(x, t)\underline{U}^2+c_1(x, t)\underline{U}G_1*\overline{U}
\nonumber	\\
	&=\mathcal{L}_1[W]+h_1(x, t)W+h_2(x, t)G_1*W,
	\end{align}
where
\begin{align*}
&h_1(x, t)=a_1(x, t)-b_1(x, t)(\overline{U}+\underline{U})-c_1(x, t)G_1*\underline{U},\\
&h_2(x, t)=c_1(x, t)\underline{U}.
\end{align*}
Multiplying (\ref{78}) by $W$ and integrating it over the space periodic domain $D_1$, we have 
\begin{align}\label{79}
\frac{1}{2}\frac{d}{dt}\int_{D_1} W^2 dx =&\int_{D_1} h_1(x, t)W^2dx-\int_{D_1} W^2 dx \nonumber \\
&+\int_{D_1} W \int_{\mathbb{R}^N}J(y-x)W(y, t)dy dx +\int_{D_1}
h_2(x, t) W G_1*Wdx.
\end{align}

Next we  claim that
\begin{equation}\label{86}
\int_{D_1} \varphi(x)\int_{\mathbb{R}^N}J(y-x)\varphi(y, t)dydx-\int_{D_1} \varphi(x)^2 dx\leq 0
\end{equation}
for all $\varphi\in X_1$.
For $z\in \mathbb{R}^N$, let
$$\hat{J}_1(z)=\sum \limits_{i_1, i_2, \cdots, i_N \in \mathbb{N}}J(z+(i_1p_1, i_2p_2,\cdots, i_Np_N)).$$
Then by the assumption that $J(\cdot)$ is symmetric with respect to $0$, $\hat{J}_1(\cdot)$ is also symmetric with respect to $0$, and
$$
\int_{\mathbb{R}^N}J(y-x)\varphi(y)dy=\int_{D_1}\hat{J}_1(y-x)\varphi(y)dy.$$
Meanwhile,
$$\int_{\mathbb{R}^N}J(z)dz=\int_{D_1} \hat{J}_1(z)dz=1.$$
Hence
\begin{align*}
&\int_{D_1}\int_{D_1} \hat{J}(y-x)\varphi(y)\varphi(x)dydx-\int_{D_1} \varphi^2(x)dx \\
 &= \int_{D_1}\int_{D_1} \hat{J}(y-x)\varphi(y)\varphi(x)dydx-\int_{D_1}\int_{D_1}\hat{J}(y-x) \varphi^2(x)dydx \\
 &=\int_{D_1}\int_{D_1} \hat{J}(y-x) \varphi(x)(\varphi(y)-\varphi(x))dydx \\
 &= \frac{1}{2}\int\int_{D_1\times D_1}\hat{J}(y-x) \varphi(x)(\varphi(y)-\varphi(x))dydx+\frac{1}{2}\int\int_{D_1\times D_1}\hat{J}(y-x) \varphi(x)(\varphi(y)-\varphi(x))dydx\\
 &=\frac{1}{2}\int\int_{D_1\times D_1}\hat{J}(y-x) \varphi(x)(\varphi(y)-\varphi(x))dydx+\frac{1}{2}\int\int_{D_1\times D_1}\hat{J}(y-x) \varphi(y)(\varphi(x)-\varphi(y))dydx \\
 &=-\frac{1}{2}\int\int_{D_1\times D_1} \hat{J}(y-x)(\varphi(y)-\varphi(x))^2dydx \\
 &\leq 0,
\end{align*}
in which we have used the symmetric properties of $\hat{J}(\cdot)$. Then (\ref{86}) holds and hence
$$
-\int_{D_1} W^2(\cdot, t) dx+\int_{D_1} W(\cdot, t)\int_{\mathbb{R}^N}J(y-x)W(y, \cdot)dydx\leq 0
$$
 for all $t\geq 0$.
Similarly, we have
$$\int_{D_1}  W(\cdot, t) G_1*W(\cdot, t)dx\leq \int_{D_1} W^2(\cdot, t) dx.$$
This, together with (\ref{79}), implies that
$$ \frac{1}{2}\frac{d}{dt}\int_{D_1} W^2 dx
 \leq \int_{D_1} (h_1(x, t)+h_2(x,t))W^2dx.$$

Now,  by $(\textbf{A}_4)$,
$$
\sup_{x\in D_1,t\in\R} h_1(x,t)+h_2(x,t)<0.
$$
Hence there is $\alpha>0$ such that
$$ \frac{1}{2}\frac{d}{dt}\int_{D_1} W^2 dx
 \leq -\alpha \int_{D_1} W^2dx.$$
This implies that
 $$\int_{D_1} W^2 dx\rightarrow 0, \quad t\rightarrow \infty.$$
Note that $W$ is a time period continuous function, then $\underline{U}\equiv \overline{U}:=U^*$, which is a positive time periodic solution of $(\mathbf{P}_1)$.
\end{proof}

\subsection{Proof of Theorem \ref{104} (3) }

In this subsection, we prove Theorem \ref{104} (3).

\begin{proof} [Proof of Theorem \ref{104} (3)]
 In this case,  according to \cite{hale1996}, we know that the equation $$u_t=u[a_1(t)-(b_1(t)+c_1(t)) u]$$ has exactly one positive $T$-period solution $\phi^*(t)$ and for any $\varepsilon>0$, $K>0$, there exists $t_{\varepsilon, K}$ such that for $t\geq t_{\varepsilon, K}$,
            $$\hat{u}(t; K)-\varepsilon \leq \phi^*(t)\leq \hat{u}(t; K)+\varepsilon,$$
            where $\hat{u}(t; K)$ is the solution of
            \begin{equation*}
    \left\{
    \begin{aligned}
   & u_t=u[a_1(t)-(b_1(t)+c_1(t))u],\\
   &u(0)=K>0.
    \end{aligned}
    \right.
    \end{equation*}
    Obviously, $\phi^*(t)$  is also a positive time period solution of $(\mathbf{P}_1)$.
             For any $u_0\in \hat{X}_i$ with $\inf \limits_{x\in \mathbb{R}^N} u_0(x)>0$, denote $\bar{u}(0)=\sup \limits_{x \in \mathbb{R}^N} u_0(x)$ and $\underline{v}(0)=\inf \limits_{x \in \mathbb{R}^N} u_0(x)$ and
let $(\bar{u}(t),\ \underline{v}(t))=(\bar{u}(t; \bar{u}(0), \underline{v}(0))$, $\underline{v}(t; \bar{u}(0), \underline{v}(0)))$ be the solution of
\begin{equation*}
 \left\{
 \begin{aligned}
 & u_t=u[a_1(t)-b_1(t)u-c_1(t)v],\\
 & v_t=v[a_1(t)-b_1(t)v-c_1(t)u].
 \end{aligned}
 \right.
 \end{equation*}
 with initial condition
 $$(\bar{u}(0; \bar{u}(0), \underline{v}(0)), \ \underline{v}(0; \bar{u}(0), \underline{v}(0)))=(\bar{u}(0),\  \underline{v}(0)),$$
 then by Lemma \ref{21}, we have $\bar{u}(t)\geq \underline{v}(t)$ and
 \begin{equation}\label{80}
 \underline{v}(t)\leq u_1(x, t; u_0)\leq \bar{u}(t).
 \end{equation}
  Thus, we have
 \begin{equation*}
 \left\{
 \begin{aligned}
 & \bar{u}_t=\bar{u}[a_1(t)-b_1(t)\bar{u}-c_1(t)\underline{v}]\geq \bar{u}[a_1(t)-b_1(t)\bar{u}-c_1(t)\bar{u}],\\
 & \underline{v}_t=\underline{v}[a_1(t)-b_1(t)\underline{v}-c_1(t)\bar{u}]\leq \underline{v}[a_1(t)-b_1(t)\underline{v}-c_1(t)\underline{v}],
 \end{aligned}
 \right.
 \end{equation*}
 which means $\bar{u}(t)\geq \hat{u}(t; \bar{u}(0))$ and $\underline{v}(t)\leq \hat{u}(t; \underline{v}(0))$. Then for any $\varepsilon>0$, $\bar{u}(0)>0$ and $\underline{v}(0)>0$, there exists $t_{\varepsilon, \bar{u}(0), \underline{v}(0)}=\max \{t_{\varepsilon, \bar{u}(0)}, \ t_{\varepsilon, \underline{v}(0)}\}$ such that for $t\geq t_{\varepsilon, \bar{u}(0), \underline{v}(0)}$,
\begin{equation}\label{81}
\underline{v}(t)-\varepsilon\leq \hat{u}(t; \underline{v}(0))-\varepsilon \leq \phi^*(t)\leq \hat{u}(t; \bar{u}(0))+\varepsilon\leq \bar{u}(t)+\varepsilon.
\end{equation}

Now we claim that $0\leq\ln\frac{\bar{u}(t)}{\underline{v}(t)}\rightarrow 0$ as $t$ goes to $\infty$. Note that
\begin{equation*}
\frac{d}{dt}\ln{\frac{\bar{u}(t)}{\underline{v}(t)}}=\frac{\bar{u}'(t)}{\bar{u}(t)}-\frac{\underline{v}'(t)}{\underline{v}(t)}= -(b_1(t)-c_1(t))(\bar{u}-\underline{v})\leq -\inf \limits_{t\in[0, T]}\{b_1(t)-c_1(t)\}(\bar{u}-\underline{v}),
\end{equation*}
and $b_{1L}>c_{1M}$, we know that $\inf \limits_{t\in[0, T]}\{b_1(t)-c_1(t)\} \geq b_{1L}-c_{1M}>0$.
Note that $-(a-b)\leq -b\ln\frac{a}{b}$ for $a>b>0$, then we have
$$\frac{d}{dt}\ln{\frac{\bar{u}(t)}{\underline{v}(t)}}\leq -\inf \limits_{t\in[0,\ T]}\{b_1(t)-c_1(t)\}\underline{v}(t)\ln{\frac{\bar{u}(t)}{\underline{v}(t)}}.$$ Then
$\frac{d}{dt}\ln{\frac{\bar{u}(t)}{\underline{v}(t)}}\leq -K_0\ln{\frac{\bar{u}(t)}{\underline{v}(t)}}$, where
$$K_0=\inf \limits_{t\in[0, T]}\{b_1(t)-c_1(t)\}\inf \limits_{t\geq 0} \underline{v}(t)>0.$$
Thus,
\begin{equation}\label{82}
0\leq \ln{\frac{\bar{u}(t)}{\underline{v}(t)}}\leq \ln{\frac{\bar{u}(0)}{\underline{v}(0)}}\exp(-K_0t)\rightarrow 0 , \quad t\rightarrow \infty.
\end{equation}
It follows from (\ref{80})-(\ref{82}) that (\ref{83}) holds.
\end{proof}
\section*{Acknowledgments}
Jianping Gao would like to thank the China Scholarship Council of China (201706130064)
for financial support during the period of his overseas study and to express his gratitude to the Department of Mathematics and Statistics, Auburn University for its kind hospitality. Shangjiang Guo is partially supported by NSF of China (11671123).

\end{document}